\documentclass[a4paper,10pt]{amsart}
\usepackage[utf8]{inputenc}
\usepackage{amsmath,amssymb,amsthm}
\usepackage{stmaryrd}
\usepackage{xcolor}
\usepackage{enumerate}
\usepackage{caption}
\usepackage[title]{appendix}

\usepackage{hyperref}

\theoremstyle{plain}

\newtheorem{thm}{Theorem}[section]
\newtheorem{cor}[thm]{Corollary}
\newtheorem{lem}[thm]{Lemma}
\newtheorem{prop}[thm]{Proposition}

\newtheorem{thmABC}{Theorem}

\theoremstyle{definition}

\newtheorem{defn}[thm]{Definition}
\newtheorem{rmk}[thm]{Remark}
\newtheorem{ex}[thm]{Example}

\newcommand{\Ind}{\operatorname{Ind}}
\newcommand{\Res}{\operatorname{Res}}
\newcommand{\Hom}{\operatorname{Hom}}
\newcommand{\End}{\operatorname{End}}

\newcommand{\Out}{\operatorname{Out}}
\newcommand{\bra}{\llbracket}
\newcommand{\ket}{\rrbracket}
\newcommand{\Irr}{\mathrm{Irr}}

\newcommand{\GL}{\operatorname{GL}}
\newcommand{\SL}{\operatorname{SL}}

\newcommand{\C}{\mathbb{C}}
\newcommand{\Q}{\mathbb{Q}}
\newcommand{\F}{\mathbb{F}}
\newcommand{\Z}{\mathbb{Z}}
\renewcommand{\Re}{\mathrm{Re}}
\renewcommand{\epsilon}{\varepsilon}

\numberwithin{equation}{section}

\title[Weil zeta functions of group representations over finite fields]{Weil 
zeta functions of group representations \\ over finite fields}
\author{Ged Corob Cook} 
\address{University of Lincoln, Brayford Pool, Lincoln LN6 7TS, UK}
\email{gcorobcook@gmail.com}
\author{Steffen Kionke}
\address{FernUniversit\"at in Hagen, Fakult\"at f\"ur Mathematik und Informatik, 58084 Hagen, Germany}
\email{steffen.kionke@fernuni-hagen.de}
\author {Matteo Vannacci}
\address{Matematika Saila \\ 
UPV-EHU \\ Barrio Sarriena, s/n \\
48910 Bilbao, Spain}
\email{matteo.vannacci@ehu.eus}

\date{\today}

\thanks{The first author was funded by the Engineering and Physical Sciences Research Council [grant number EP/V036874/1]. The second author was funded by the Deutsche Forschungsgemeinschaft (DFG, German Research Foundation) - 441848266. The third author has 
been supported by the Spanish Government, grants PID2020-117281GB-I00 
and PID2019-107444GA-I00, partly with FEDER funds, and the Basque Government, 
grant IT1483-22.}

\begin{document}
\renewcommand{\labelenumi}{(\roman{enumi})}
\renewcommand{\labelenumii}{(\alph{enumii})}

\begin{abstract}
In this article we define and study a zeta function $\zeta_G$ -- similar to the Hasse-Weil zeta function -- which enumerates absolutely irreducible representations over finite fields of a (profinite) group $G$. The zeta function converges on a complex half-plane for all UBERG groups and admits an Euler product decomposition. Our motivation for this investigation is the observation that the reciprocal value $\zeta_G(k)^{-1}$ at a positive integer $k$ coincides with the probability that $k$ random elements generate the completed group ring of $G$. The explicit formulas obtained so far suggest that $\zeta_G$ is rather well-behaved. 

A central object of this article is the abscissa of convergence $a(G)$ of $\zeta_G$. We calculate the abscissae for free abelian, free abelian pro-$p$, free pro-$p$, free pronilpotent and free prosoluble groups. More generally, we obtain bounds (and sometimes explicit values) for the abscissae of free pro-$\mathfrak{C}$ groups, where $\mathfrak{C}$ is a class of finite groups with prescribed composition factors. We prove that every real number $a \geq 1$ is the abscissa $a(G)$ of some profinite group $G$. 

In addition, we show that the Euler factors of $\zeta_G$ are rational functions in $p^{-s}$ if $G$ is virtually abelian. For finite groups $G$ we calculate $\zeta_G$ using the rational representation theory of $G$.
\end{abstract}

\maketitle

\tableofcontents

\section{Introduction}
The use of zeta functions to study asymptotic properties of algebraic objects is a well-established 
branch of algebra, see for instance \cite{duSautoySegal, GSS, LS, Voll1,  Weil} and references therein. In particular, there is considerable interest in zeta functions enumerating all \cite{AKOV,LL}, some \cite{KiKl} or equivalence classes of \cite{SV} complex irreducible representations of infinite groups. Every time the counting problem is adapted to the class of groups under investigation, which gives the theory of representation growth a piecemeal flavour. In this article we take a different approach suitable for the large class of UBERG groups: we define and study a \emph{zeta function} enumerating absolutely irreducible representations over finite fields.

For a profinite group $G$ and a field $F$, we write $r^*(G,F,n)$ to denote the number of absolutely irreducible representations of $G$ of dimension $n$ defined over $F$.
We say that $G$ has \emph{UBERG} if there exists a 
positive constant $c>0$ such that $r^*(G,F,n)\le \vert F \vert^{cn}$ for every 
finite field $F$. UBERG stands for `uniformly bounded exponential 
representation growth' (over finite fields) and, maybe surprisingly, it shows up 
naturally in the study of probabilistic generation properties of profinite 
groups. In fact, a finitely presented profinite group is \emph{positively 
finitely related} (PFR) exactly if it has UBERG. Moreover, a profinite group 
has UBERG if and only if the completed group algebra 
$\hat{\mathbb{Z}}\llbracket G\rrbracket$ is \emph{positively finitely generated} 
(PFG, see \cite{KV}) as a $\hat{\mathbb{Z}}\llbracket G\rrbracket$-module.
The properties of UBERG groups were the central object of the authors' paper 
\cite{CCKV}.

We define the following complex function for an UBERG profinite group $G$:
\begin{equation}\label{eq:defn_zeta}
	\zeta_G(s) := \exp  \left( \sum_{p \in \mathcal{P}}\sum_{n=1}^\infty\sum_{j=1}^\infty \frac{r^\ast(G,\F_{p^j},n)}{j} p^{-snj}|\mathbb{P}^{n-1}(\F_{p^j})|  \right) \text{\quad for } s\in \mathbb{C}.
\end{equation}
We will see that, for an UBERG group $G$, the above sum converges on some complex half-plane (cf.\ Corollary~\ref{cor:conv}). It admits an obvious Euler product decomposition $\zeta_G(s) = \prod_{p\in \mathcal{P}} \zeta_{G,p}(s)$. The formula for $\zeta_G$ is reminiscent of the Hasse-Weil zeta function of an algebraic variety $V$ where absolutely irreducible representations of $G$ over $\F_{q}$ now take the role of $\F_q$-rational points of $V$. The main objective of this article is to start the investigation of the function $\zeta_G(s)$ and the group-theoretical properties of $G$ detected by it. 

If $G$ is an abstract group, any finite-dimensional representation of $G$ over a finite field has to factor through a finite quotient; hence, for a group $G$ with UBERG profinite 
completion, we define $\zeta_G(s) := \zeta_{\widehat{G}}(s)$.
For instance, for $G = \Z$ one obtains (see Example~\ref{Z})
\[
	\zeta_{\Z}(s) = \frac{\zeta(s-1)}{\zeta(s)},
\]
where $\zeta$ denotes Riemann's zeta function. In general, $\zeta_G$ is the Hasse-Weil zeta function of $\mathrm{Spec(\Z[G])}$ whenever $G$ is a finitely generated abelian group.
More examples can be found in Appendix \ref{app:A}; our list of examples conveys the impression that $\zeta_G$ is rather well-behaved.
It is worth noting that if one replaces absolutely irreducible by irreducible representations in the formula for the
zeta function the resulting zeta function is less well-behaved;  see Appendix~\ref{app:B}. 

Our initial motivation to study $\zeta_G$ is a significant connection
between the values of the zeta function $\zeta_G(s)$ at the natural numbers and 
the probability $P_R(R,k)$ that $k$ random elements generate the group ring 
$R=\widehat{\Z}\bra G \ket$ of $G$ over $\widehat{\Z}$ as a $\widehat{\Z}\bra G \ket$-module. 
In this setting, we have the following theorem (see 
Section~\ref{sec:prob}).

\begin{thmABC}\label{thmABC:prob}
 Let $G$ be an UBERG profinite group and let be $R=\widehat{\Z}\bra G \ket$ its group ring over $\widehat{\Z}$. For all sufficiently large integers $\ell$ the following equality holds:
\[
	P_R(R,\ell)^{-1} = \zeta_G(\ell).
\]
\end{thmABC}

Following the tradition of the study of zeta functions associated to algebraic 
objects, we continue by studying the \emph{abscissa of convergence} of the 
function $\zeta_G(s)$, i.e.\ the real number
\begin{equation}\label{eq:abscissa_defn}
   a(G) = \inf\{t\in \mathbb{R}_{\ge 1} \mid \zeta_G(t) \text{ converges}\}.
\end{equation}

In Section~\ref{sec:abscissa} we prove the following properties of the 
function $a(G)$.

\begin{thmABC}\label{thmABC:abscissa}
 Let $G$ be an UBERG profinite group.
\begin{enumerate}[(i)]
 \item\label{it:open-bound} If $H$ is an open subgroup of $G$, then $$\frac{a(H)}{\vert G:H \vert} \le 
a(G)\le a(H) +1 - \frac{1}{\vert G:H \vert}.$$ Moreover, if $G=H\times F$ for a 
finite group $F$, then $a(G)=a(H)$. 
 \item If $K$ a closed normal subgroup of $G$, then $a(G/K) \le a(G)$. 
Moreover, if the quotient map $G\to G/K$ splits, then $$a(G) \le a(K) + a(G/K) 
+1.$$
\end{enumerate}
\end{thmABC}
As a corollary we obtain that $a(G)=1$ for all finite groups $G$, however there are infinite groups with abscissa $1$ (see Theorem~\ref{thmABC:examples} and \ref{thmABC:arbitrary}). We show in Example \ref{ex:free_prop} that the lower bound in \eqref{it:open-bound} is sharp, but it appears to us that the upper bound in \eqref{it:open-bound} is not sharp. In order to obtain a better understanding, we construct examples of groups $G$ that are split 
extensions of a finite index normal subgroup $N$ such that $a(G)> a(N)$ in 
Section~\ref{largeabscissae}.

We move on to calculate the abscissa of convergence of various families of 
profinite groups in Section \ref{sec:examples}. We recall briefly the notation used in the theorem. As usual, 
$F^{sol}_r$, $F^{nil}_r$ and $F^{p}_r$ denote the free prosoluble, pronilpotent 
and pro-$p$ groups on $r$ generators, respectively, and $\Z_p$ denotes the $p$-adic 
integers. Finally, $c_{nil} = \frac{5\log(2)}{2\log(3)}$ and $c_{sol} = \frac{2}{3}+\frac{5\log(2)}{2\log(3)}$ are the constants from \cite{Wolf} and \cite{Palfy}, respectively.

\begin{thmABC}
	\phantomsection\label{thmABC:examples} Let $p$ be a prime number and $r\in \mathbb{N}$.
	\begin{enumerate}[(i)]\setlength\itemsep{1em}
	\item $\displaystyle a(\mathbb{Z}^r) = r+1$.
	\item $\displaystyle \frac{r}{K(p)} \leq a(\Z_p^r) \leq \frac{r-1}{K(p)} + 1.$ In particular, $a(\Z_p) = 1$.
	\item $\displaystyle a(F^p_r) = \displaystyle \frac{r-1}{K'(p)} + 1$ if $p$ is odd and $\displaystyle a(F^2_r) = c_{nil}(r-1) +1$. %
	\item $\displaystyle a(F^{nil}_r) =c_{nil}(r-1) +1$ if $r>2$ and $\displaystyle a(F^{nil}_2) = 3$.     \item $\displaystyle a(F^{sol}_r) = c_{sol}(r-1)+1$.
    \item If $G$ is a pronilpotent group of finite rank $r$, then $a(G)\le r+1$.
   \end{enumerate}
\end{thmABC}

Here $K(p)$ and $K'(p)$ are number-theoretic constants defined in Sections \ref{sec:free_ab_pro_p} and \ref{sec:free-pro-p} respectively, with $1 \leq K(p) \leq 2.1115$, $K'(p) \leq K(p)$, and
$$K'(p) \begin{cases} = \frac{2 \log(3)}{5 \log(2)} 
& \text{ if } p=2; \\
= \frac{(p-1)\log(p+1)}{p\log(p)}<1 & \text{ if $p$ is a Mersenne prime;} \\
\geq 1 & \text{ otherwise.}
\end{cases}$$
It is worth pointing out that the calculation of the abscissa of $\Z_p^r$ touches on 
some deep number-theoretic conjectures and results concerning the distribution of primes in arithmetic progressions: see 
Section~\ref{sec:free_ab_pro_p}. Strong results are needed to show that $a(\Z_p^2) > 1$.

On the other hand, we can estimate (and sometimes calculate exactly) the absissa of some free groups with restricted non-abelian composition factors of \emph{large degree}, as we explain now. Let $\mathfrak{C}$ be a NE-formation of finite groups containing the cyclic groups of prime order. Let $F^{\mathfrak{C}}_r$ be the free pro-$\mathfrak{C}$ group on $r$ generators. If $\mathfrak{C}$ contains alternating groups of arbitrarily large degree, or classical groups with natural representation of arbitrarily large dimension, one can show that $F^{\mathfrak{C}}_r$ does not have UBERG (cf.\  Theorem~\ref{thm:lowerbound-general}). Assume it does not. Let $c_0$ be maximal such that $Alt(c_0) \in \mathfrak{C}$, and assume that $c_0$ is large. In Theorem~\ref{dominant-abscissa} we show that the abscissa of $F^\mathfrak{C}_r$ is dominated by representations of a very specific type. Using this, we can bound the absissa of $F^\mathfrak{C}_r$.

\begin{thmABC}
Let $\mathfrak{C}$ be a NE-formation of finite groups containing the cyclic groups of prime order and suppose that $\mathfrak{C}$ does not contain alternating groups of degree greater than $c_0$. Let $F^{\mathfrak{C}}_r$ be the free pro-$\mathfrak{C}$ group on $r$ generators. Then:
 $$c_{\mathfrak{C}, \text{space}}(r-1)+1 \leq a(F^{\mathfrak{C}}_r) \leq \max(c_{\mathfrak{C}, \text{space}}(r-1)+1, c_{\mathfrak{C}, \text{time}}(r-1)+2).$$
\end{thmABC}

The constants $c_{\mathfrak{C},space}$ and $c_{\mathfrak{C},time}$ are defined just before Theorem~\ref{spacetime-bounds}. Note that $c_{\mathfrak{C},time} \leq c_{\mathfrak{C},space}$, so these bounds differ by at most $1$. The inspiration for these names is that the convergence of $\zeta_G$ is often dominated either by the representations in all dimensions over a specific field, or by the representations in a specific dimension over all fields; we think of the former behaviour as `space-like' and the latter as `time-like'. Not all groups fit into one of these classes, and the time-like groups could be broken down further into being dominated by representations in a specific characteristic, versus in a specific power of the characteristic; but at any rate, this categorisation is useful for the groups we have studied. For a surprising example, $F^{nil}_r$ is space-like for $r>2$ and time-like for $r=2$, which is what causes the behaviour of the abscissa in Theorem \ref{thmABC:examples}(iv).

Additionally, we can calculate the absissa of $F^\mathfrak{C}_r$ for two specific NE-formations $\mathfrak{C}$ where the zeta-function is space-like: the NE-formation $\mathfrak{C}_{\mathrm{Alt}(c_0)}$ generated by all cyclic groups of prime order and the alternating groups of degree at most $c_0$; and the NE-formation $\mathfrak{C}_{\Sigma(c_0)}$ generated by all the simple groups in $\mathfrak{C}_{Alt(c_0)}$, all the sporadic groups, all the exceptional groups of Lie type, and all the classical simple groups whose natural representation has dimension at most $c_0$. In these cases,
\[
   a(F^{\mathfrak{C}_{Alt(c_0)}}_r) = \frac{c_0\log_2(c_0!)}{(c_0-\delta(c_0))(c_0-1)}(r-1)+1
\]
where $\delta(c_0)=1$ or $2$ for $c_0$ odd or even, respectively, and
\[
a(F^{\mathfrak{C}_{\Sigma(c_0)}}_r) = (c_0+\frac{(c_0-1)\log(\prod_{i=1}^{c_0}(1-2^{-i}))+\log(c_0!)}{c_0(c_0-1)\log(2)})(r-1)+1.
\]

In general determining the abscissa of convergence for a profinite group seems to be a hard problem.
One might wonder whether all real numbers  $\alpha \geq 1$
are the abscissa of some profinite group. We answer this positively in Section \ref{sec:arbitrary}. 
\begin{thmABC}\label{thmABC:arbitrary}
Let $\alpha$ be a non-negative real number and set $$G_\alpha = \prod_{p \in 
\mathcal{P}} \mathrm{SL}(2,p)^{\lfloor p^\alpha \rfloor}.$$
Then $\zeta_{G_\alpha}(s)$ has abscissa of convergence 
$\alpha/2+1$.
\end{thmABC}

Finally, we concentrate our attention on analytic properties of the function 
$\zeta_G(s)$ in special cases. In the case where $G$ is a finite group, we can give in Section \ref{sec:finite} an explicit formula for the zeta function, depending on the rational representations of $G$, up to rational functions. First some notation: for a meromorphic function $f \colon \C \to \C$ and a natural number $n$ we define
\[
	f^{\# n}(s) := \prod_{j = 0}^{n-1} f(ns-j).
\]
We write $f \sim g$, if there exists $K\in \mathbb{N}$ such that $f/g$ is of 
the form
\begin{equation}\label{eq:rational}
	\prod_{k=1}^K (1- M_k^{-a_ks+b_k})^{\varepsilon_k}	
\end{equation}
with $M_k \in \mathbb{N}_{\geq 2}$, $a_k > b_k \in \mathbb{N}_0$ and 
$\varepsilon_k \in \{\pm 1\}$, for $k=1,\ldots,K$.

Let $G$ be a finite group and let $V$ be an irreducible rational representation 
with character $\chi$.
We write $K_\chi$ for the center of the endomorphism algebra of $\End_G(V)$ and 
$m(\chi)$ denotes the Schur index of $\chi$. We note that $K_\chi$ is an 
algebraic number field and
$\dim_{K_\chi} \End_G(V) = m(\chi)^2$.

\begin{thmABC}\label{thmABC:finite_groups}
 	Let $G$ be a finite group. Then
	\[
		\zeta_G(s) \sim \prod_{\chi \in \Irr(G,\Q)} \zeta_{K_\chi}^{\# 
n_\chi}(s).
 	\]
	where $\zeta_{K_\chi}$ denotes the Dedekind zeta function of $K_{\chi}$ and 
$n_\chi = \frac{\chi(1)}{[K_\chi:\Q] m(\chi)}$.
	The zeta function $\zeta_G(s)$ admits a meromorphic continuation to $\C$,  
it has a pole of order $|\Irr(G,\Q)|$ at $s = 1$ and all other poles are 
located 
at rational numbers in the interval $[0,1-\sqrt{|G|^{-1}}]$.
\end{thmABC}

Since our zeta function bears a resemblance to the Hasse-Weil zeta function 
of an algebraic variety, it is a very natural question to ask if (or when) the local factors $\zeta_{G,p}$  satisfy parts of Weil's conjectures. In Section \ref{sec:rationality} we show the following.

\begin{thmABC}\label{thmABC:rational}
 Let $p$ be a prime and let $G$ be a finitely generated virtually abelian 
group. Then $\zeta_{G,p}(s)$ is a rational function in $p^{-s}$.
\end{thmABC}

The proof uses the \emph{moduli varieties} $M_n$ of $n$-dimensional absolutely 
irreducible representations and the classical solution of the Weil conjectures. 
We suspect that the zeta functions of a larger class of UBERG groups have 
rational local factors. To support this, we show in Appendix~\ref{app:A} that 
the local factors of the zeta function of the lamplighter groups $C_2\wr \Z$ 
and $C_3\wr\Z$ are rational.

\vspace*{1em}

\noindent\textit{Notation.}
As it is customary when working with profinite groups and rings, all subgroups will be 
assumed to be closed and all homomorphisms will be assumed to be continuous.

\section{The zeta function \`a la Weil of PFG profinite rings}\label{sec:zeta_def}

Let $R$ be a profinite ring. We say that $R$ is PFG if it is positively finitely generated as left module over itself.
Let $F$ be a finite field, and let $\overline{F}$ denote the algebraic closure. An $F\otimes_\mathbb{Z} R$ module $M$ is \emph{absolutely simple} if it is simple and $\overline{F}\otimes_F M$ is simple.
We define $r^\ast(R,F,n)$ to be the number of isomorphism 
classes of absolutely simple $F \otimes_\mathbb{Z} R$ modules of dimension $n$ 
over $F$.
We observe that if $R = \widehat{\mathbb{Z}}\bra G\ket$ is the completed group ring of a profinite group $G$, then
$r^\ast(R,F,n) = r^\ast(G,F,n)$.

\begin{lem}\label{lem:pfg-ring-characteristion}
A profinite ring $R$ is PFG if and only if there is a constant $c > 0$ such that 
for all finite fields $F$ and all $n$ the following inequality holds
\[
	r^\ast(R,F,n) \leq |F|^{cn}.
\]
\end{lem}
\begin{proof}
The proofs of \cite[Prop.~6.1]{KV} and \cite[Lem.~6.8]{KV} go through without changes in our current situation.
\end{proof}

Suppose $R$ is a PFG profinite ring. Let $c_0(R)$ be the infimum of the numbers 
$c$ such that, for all but finitely many tuples $(p,j,n)$, $r^\ast(R,\F_{p^j},n) 
\leq p^{cnj}$. We may define $c_0(R) = \infty$ if $R$ is not PFG.

\begin{lem}
	\label{generalbounds}
	The series 
	\[
	\sum_{p \in \mathcal{P}}\sum_{n=1}^\infty\sum_{j=1}^\infty \frac{r^\ast(R,\F_{p^j},n)}{j} p^{-snj}|\mathbb{P}^{n-1}(\F_{p^j})|
	\]
	converges absolutely for all complex numbers $s = \sigma + i \tau$ with $\sigma > c_0(R)+1$, and diverges for $s = \sigma + i \tau$ with $\sigma < c_0(R)$.
\end{lem}
\begin{proof}
	Since all the coefficients are real and positive, absolute convergence is the same as convergence, and we may assume $s$ is real and rearrange the terms. Let $\varepsilon > 0$. Replacing $p^j$ with $q$, we get the upper bound 
	\[
	\sum_{p \in \mathcal{P}}\sum_{n=1}^\infty \sum_{j=1}^\infty \frac{r^\ast(R,\F_{p^j},n)}{j} p^{-snj}|\mathbb{P}^{n-1}(\F_{p^j})| \le \kappa \sum_{q=1}^\infty\sum_{n=1}^\infty q^{n(c_0(R)+\varepsilon)} q^{-\sigma n}q^{n-1}.
	\]
	for a suitable constant $\kappa > 0$.
	 When $\sigma > c_0(R)+1+2\varepsilon$, we have 
	 \[\sum_{n=1}^\infty q^{n(c_0(R)+\varepsilon)} q^{-\sigma n}q^{n-1} = \frac{q^{c_0(R)+\varepsilon-\sigma}}{1-q^{c_0(R)+\varepsilon+1-\sigma}} \ll q^{-1-\varepsilon},\]
	  so the whole sum converges.
	
	On the other hand, when $\sigma < c_0(R)$, fix $\varepsilon >0$ such that $\sigma + \varepsilon < c_0(R)$. For infinitely many tuples $(p,j,n)$ we have $r^\ast(R,\F_{p^j},n) > p^{(\sigma + \varepsilon) nj}$; call the set of such tuples $S$. So 
	
	$$\sum_{p \in \mathcal{P}}\sum_{n=1}^\infty\sum_{j=1}^\infty \frac{r^\ast(R,\F_{p^j},n)}{j} p^{- \sigma nj}|\mathbb{P}^{n-1}(\F_{p^j})| > \sum_{(p,j,n) \in S} p^{\varepsilon j}/j.$$
	
	Since $p \geq 2$, a calculation shows $p^{\varepsilon j}/j \geq \varepsilon \log(2) 2^{1/\log(2)}$, so the sum diverges. Hence, by standard results on Dirichlet series, the sum diverges for all $s$ with $\sigma < c_0(R)$.
\end{proof}

\begin{cor}\label{cor:conv}
If $R$ is a PFG profinite ring, then 
\[
	\zeta_R(s) := \exp  \left( \sum_{p \in \mathcal{P}}\sum_{n=1}^\infty\sum_{j=1}^\infty \frac{r^\ast(R,\F_{p^j},n)}{j} p^{-snj}|\mathbb{P}^{n-1}(\F_{p^j})|  \right)
\]
defines a holomorphic function in some right half plane.
\end{cor}
We will call $\zeta_R(s)$ the zeta function of $R$. Now we compute some 
examples that will be useful in the proof of Theorem~\ref{thmABC:prob}.
\begin{ex}
\label{trivial}
Let $R= \widehat{\mathbb{Z}}= \widehat{\mathbb{Z}}\bra1\ket$. Let $F$ be any finite field. Then $\widehat{\mathbb{Z}}$ has exactly one one-dimensional absolutely simple module over every finite field.
Therefore (using the logarithmic series)
\[
	\zeta_{\widehat{\mathbb{Z}}}(s) = \exp \left( \sum_{p \in \mathcal{P}}\sum_{j=1}^\infty \frac{1}{j} p^{-sj}\right)
					= \prod_{p\in \mathcal{P}} 
\left(1-\frac{1}{p^s}\right)^{-1} = \zeta(s)
\]
is the Riemann zeta function.
\end{ex}

\begin{ex}
\label{Z}
Let $R= \widehat{\mathbb{Z}}\bra \hat{\mathbb{Z}} \ket$. Let $F$ be any finite field. $\widehat{\mathbb{Z}}\bra \hat{\mathbb{Z}} \ket$ has $q-1$ one-dimensional absolutely simple modules over $\mathbb{F}_q$.
Therefore
\[
	\zeta_{R
	}(s) = \exp \left( \sum_{p \in \mathcal{P}}\sum_{j=1}^\infty \frac{p^j - 1}{j} p^{-sj}\right)
					= \prod_{p\in \mathcal{P}} 
\left(1-\frac{1}{p^{s-1}}\right)^{-1} \left(1-\frac{1}{p^s}\right) = 
\frac{\zeta(s-1)}{\zeta(s)}
\]
\end{ex}
\begin{ex}\label{ex:zeta-matrix-algebra}
Let $p$ be a fixed prime.
Let $R= M_n(\F_{p^k})$ be the matrix algebra over the finite field $\F_{p^k}$.
The ring has $k$ absolutely simple modules of dimension $n$ over $F$ if and only if $F$ contains $\F_{p^k}$, i.e.,
\[
	r^\ast(R,F,m) = \begin{cases}	k & \text{ if } m=n \text{ and } \F_{p^k} \subseteq F; \\
	0 & \text{ otherwise.}
	\end{cases}
\]
The zeta function is
\[
	\zeta_{ R }(s) = \exp \left( \sum_{j=1}^\infty \frac{k}{jk} p^{-sjkn}|\mathbb{P}^{n-1}(\F_{p^{kj}})|\right) = 
	\prod_{i=0}^{n-1} \left(1-\frac{p^{ki}}{p^{skn}}\right)^{-1}.
\]
\end{ex}

Next we show that the zeta function of a profinite ring is the pointwise limit of the 
zeta functions of its finite quotients.

\begin{prop}[Continuity]\label{prop:continuity}
Let $R = \varprojlim_{i \in \mathbb{N}} R_i$ be a PFG profinite ring which is given a an inverse limit of finite rings $R_i$.
If $\zeta_R(s)$ converges absolutely at $s\in \mathbb{C}$, then 
\[
	\zeta_R(s) = \lim_{i\to \infty} \zeta_{R_i}(s).
\]
\end{prop}
\begin{proof}
Since every absolutely simple module factors over some $R_i$, we have
\[
	r^\ast(R,F,n) = \lim_{i\to\infty} r^\ast(R_i,F,n)
\]
for all finite fields $F$ and the sequence $r^\ast(R_i,F,n)$ is monotonically increasing.
By assumption there is $c > 0$ such that
\[
   r^\ast(R_i,F,n) \leq r^\ast(R,F,n) \leq |F|^{nc}
\]
for all finite fields $F$ and every $n$.
Let $s \in \mathbb{C}$ and assume that the series defining $\zeta_R(s)$ 
converges absolutely at $s$, then
\[
\lim_{i \to \infty } \zeta_{R_i}(s) = \zeta_R(s)
\]
by Lebesgue's theorem of dominated convergence (e.g., 
\cite[5.6~Thm.]{bartle}).
\end{proof}

We end this preliminary section with a formula for the zeta function of the 
direct product of rings and of groups. 

\begin{prop}[Products of rings]\label{prop:products}
Let $R$, $S$ be two PFG profinite rings. Then $R \times S$ is PFG and 
\[
	\zeta_{R\times S}(s) = \zeta_R(s) \cdot \zeta_S(s)
\]
for all $s \in \mathbb{C}$ where  $\zeta_R(s)$ , $\zeta_S(s)$ are defined.
\end{prop}

Recall that, given two functions $f,g:\mathbb{N}\to \mathbb{N}$ the \emph{Dirichlet convolution} of $f$ and $g$ is the function
\[
  (f\ast g)(n) = \sum_{d\mid n} f(d) g(n/d).
\]

Recall that if $G$ and $H$ have UBERG, then the direct product $G \times H$ does too; see \cite[Theorem 6.4]{KV}.
\begin{prop}[Products of groups]
Let $G$, $H$ be two profinite groups. Then $r^\ast(G \times H,F,n) = (r^\ast(G,F,-) \ast r^\ast(H,F,-))(n)$, where $\ast$ denotes Dirichlet convolution. If $G$ and $H$ have UBERG, we can write $\zeta_{G \times H}(s)$ as
\[
\exp \left( \sum_{p \in \mathcal{P}}\sum_{n=1}^\infty\sum_{j=1}^\infty \frac{(r^\ast(G,\F_{p^j},-) \ast r^\ast(H,\F_{p^j},-))(n)}{j} p^{-snj}|\mathbb{P}^{n-1}(\F_{p^j})| \right).
\]
\end{prop}

Note that the above formula does not directly give a nice formula for the zeta function of $G\times H$ and, even for finite groups, it might be hard to compute it.

\section{Zeta functions and probability}\label{sec:prob}
In this section we will prove Theorem~\ref{thmABC:prob}.
\begin{proof}[Proof of Theorem~\ref{thmABC:prob}]
Let $R$ be a PFG profinite ring and let $J(R)$ denote the Jacobson radical. A 
tuple of elements generates $R$ if and only if it generates $R/J(R)$.
Hence $P_{R}(R,\ell) = P_{R/J(R)}(R/J(R),\ell)$ for all $\ell \in \mathbb{N}$.
Similarly, the Jacobson radical annihilates all simple modules, hence
\[
r^\ast(R,F,n) = r^\ast(R/J(R),F,n).
\]
So we may replace $R$ with $R/J(R)$, and assume that $R$ is semisimple, i.e., a 
product of matrix algebras over finite fields. Since $R$ is PFG, it has 
countably many simple quotients, so this product is countably based.

Thus we may write
$R = \varprojlim_{i \in \mathbb{N}} R_i$ as an inverse limit of finite 
semisimple rings.
Then 
$P_R(R,\ell) = \lim_{i \in \mathbb{N}} P_{R_i}(R_i,\ell)$
and so by Proposition~\ref{prop:continuity} it is sufficient to prove the theorem for finite 
semisimple rings. Since probabilities and zeta functions behave well under 
finite direct products of rings (see Proposition~\ref{prop:products}), it is sufficient to prove the equality 
for a single matrix algebra.
The probability that $\ell$ random elements generate $M_n(\F_{p^k})$ is the 
probability
that a random $(n\times n\ell)$ matrix has rank $n$, i.e.,
\[
P_{M_n(\F_{p^k})}(M_n(\F_{p^k}),\ell) = p^{-  k \ell 
n^2}\prod_{i=0}^{n-1}(p^{k\ell n}-p^{k i}) = \prod_{i=0}^{n-1}\Bigl(1-\frac{p^{k 
i}}{p^{\ell k n}}\Bigr),
\]
and the assertion follows from Example \ref{ex:zeta-matrix-algebra}.
\end{proof}

\section{Abscissae of convergence}\label{sec:abscissa}

Let $G$ be an UBERG profinite group. Write $a(G)$ for the abscissa of 
convergence of $\zeta_G(s)$, as defined in \eqref{eq:abscissa_defn}. Observing 
that all the coefficients are real and positive, we deduce that the abscissa of 
convergence of $\zeta_G(s)$ equals its abscissa of absolute convergence. 
Therefore, we can rearrange terms in $\zeta_G(s)$ and only study convergence 
for real $s$.

We start by comparing the abscissa of converge of a quotient with that of the group. The following lemma is clear.

\begin{lem}
	\label{quotient}
	For any profinite group $G$ and $N \unlhd G$, we have $a(G/N) \leq a(G)$.
\end{lem}

In particular, $a(G) \geq 1$ for any $G$, because $a(\{1\})=1$ by Example 
\ref{trivial}. The next lemma will be used repeatedly for calculations of 
abscissae.

\begin{lem}
	\label{sameabscissa}
	Suppose $f: \mathcal{P} \times \mathbb{N}_+ \times \mathbb{N}_+ \to \mathbb{R}_{\geq 1}$ is such that $f(p,n,j) = p^{o(nj)}$. For any function $g: \mathcal{P} \times \mathbb{N}_+ \times \mathbb{N}_+ \to \mathbb{R}$, the series $$A(s) = \sum_{p \in \mathcal{P}}\sum_{n=1}^\infty\sum_{j=1}^\infty \frac{g(p,n,j)}{j} p^{-snj}|\mathbb{P}^{n-1}(\F_{p^j})|$$ and $$B(s) = \sum_{p \in \mathcal{P}}\sum_{n=1}^\infty\sum_{j=1}^\infty f(p,n,j) \frac{g(p,n,j)}{j} p^{-snj}|\mathbb{P}^{n-1}(\F_{p^j})|$$ have the same abscissae of convergence.
\end{lem}
\begin{proof}
	When $A(s)$ converges, $B(s) = O_\varepsilon(A(s+\varepsilon))$ converges too for all $\varepsilon>0$, because $p^{\varepsilon nj} \geq f(p,n,j)$ for large $p,n,j$. The converse is clear.
\end{proof}

Next we deal with the abscissa of open subgroups and we prove the first item in 
Theorem~\ref{thmABC:abscissa}.

\begin{prop}
	\label{opensbgp}
	Suppose $G$ is a profinite group and $H \leq_{\mathrm{o}} G$. Then:
	\begin{enumerate}[(i)]
		\item $a(G) \leq a(H)+1-|G:H|^{-1}$;
		\item $a(H) \leq |G:H|a(G)$.
	\end{enumerate}
	If $F$ is a finite group and $G = H \times F$, then $a(G)=a(H)$.
\end{prop}
\begin{proof}
	$G$ has UBERG if and only if $H$ does, by \cite[Thm.~6.2]{KV}. If neither has UBERG, $\zeta_G$ and $\zeta_H$ have abscissa of convergence $\infty$. So we may assume they both have UBERG.

	For each absolutely irreducible $\F_{p^j}\bra G \ket$-module $N$, pick one 
	irreducible $\F_{p^j}\bra H \ket$-submodule $M$ of $\Res^G_H(N)$. We note that $M$ has the 
	structure of an  $\F_{p^{je}}\bra H \ket$-module for  
	$e$ such that $\F_{p^{je}} \cong \End_H(M)$; by \cite[(9.2)]{Isaacs} this module is absolutely irreducible and henceforth 
	we will abuse notation by writing $M$ for this $\F_{p^{je}}\bra H \ket$-module. 
	We get a map $\phi: N \mapsto M$ from absolutely irreducible $G$-modules over 
	finite fields to absolutely irreducible $H$-modules over finite fields.
	
	Fix an absolutely irreducible $\F_{p^j}\bra H \ket$-module $M$ of dimension $m$: we want to consider $\phi^{-1}(M)$. Any absolutely irreducible $G$-module $N$ in $\phi^{-1}(M)$ is defined over some field $\F_{p^{\ell}}$ with $\ell \mid  j$: note that $|\End_H(M)| \leq |\Hom_H(M,\Res^G_H(N)| = |\Hom_G(\Ind^G_H(M),N)| \leq |\End_G(N)|^{|G:H|}$, so $p^{j} \leq p^{\ell|G:H|}$ and $j/\ell \leq |G:H|$. In particular, there are at most $|G:H|$ possibilities for $\ell$. 
	Denote by $L$ the restriction of scalars from $\F_{p^{\ell}}$ to $\F_{p^j}$ of $M$. As $L$ is irreducible and $N$ appears as a quotient of $\Ind^G_H(L)$, for a fixed $\ell$ there are at most $|G:H|$ possibilities for $N \in \phi^{-1}(M)$. Hence there are at most $|G:H|^2$ possibilities for $N \in \phi^{-1}(M)$. Moreover, such an $N$ can have dimension $n$ at most $|G:H| \dim_{\F_{p^{\ell}}}L = |G:H|jm/\ell$ over $\F_{p^{\ell}}$, with $n\ell \geq mj$. 
	
	Write $\Irr^\ast(G,\F_{p^{\ell}},n)$ for the set of absolutely irreducible $\F_{p^{\ell}}\bra G \ket$-modules of dimension $n$, and define $\displaystyle\alpha(N)=\frac{p^{-sn\ell}}{\ell} \frac{p^{\ell n}-1}{p^{\ell}-1}$ for $N \in \Irr^\ast(G,\F_{p^{\ell}},n)$. Then we may write the $\log$ of the zeta function of $G$ as $$\sum_{p\in \mathcal{P}} \sum_{\ell=1}^\infty \sum_{n=1}^\infty\sum_{\Irr^\ast(G,\F_{p^{\ell}},n)} \alpha(N).$$ Grouping these terms by their image under $\phi$, this is equal to $$\sum_{p\in \mathcal{P}}\sum_{\ell=1}^\infty \sum_{n=1}^\infty\sum_{\Irr^\ast(H,\F_{p^j},m)}\sum_{N \in \phi^{-1}(M)} \alpha(N).$$ Now, for $N$ as above in $\phi^{-1}(M)$,
	\begin{align}
		\alpha(N) &= \frac{p^{-sn\ell}}{\ell}\frac{p^{\ell n}-1}{p^{\ell}-1}  = 
		\frac{p^{-sn 
				\ell}}{\ell}\frac{p^j-1}{p^{\ell}-1}\frac{p^{jm}-1}{p^j-1}\frac{p^{\ell 
				n}-1}{p^{jm}-1} \nonumber \\
		&\leq \frac{p^{-sn 
				\ell}}{\ell}\frac{jp^{j-\ell}}{\ell}\frac{p^{jm}-1}{p^j-1}\frac{\ell n p^{\ell 
				n-jm}}{jm} = \frac{n p^{-sn \ell+j-\ell+\ell n-jm}}{\ell 
			m}\frac{p^{jm}-1}{p^j-1}  \nonumber  \\
		&= \frac{nj}{\ell m} p^{(\ell n-jm)(1-s)+j-\ell}\alpha(M).
		\label{eq:ast} 
	\end{align}
	We have $j/\ell \leq |G:H|$ and $n/m \leq |G:H|j/\ell \leq |G:H|^2$, so for $s>1$, $\alpha(N) \leq |G:H|^3p^{(\ell n-jm)(1-s)+j-\ell}\alpha(M) \leq |G:H|^3p^{(1-|G:H|^{-1})j}\alpha(M)$, and hence $\sum_{N \in \phi^{-1}(M)} \alpha(N) \leq |G:H|^5p^{(1-|G:H|^{-1})j}\alpha(M)$. Finally, for $s > a(H)+1-|G:H|^{-1}$, we conclude that
	\begin{align*}
		\log(\zeta_G)(s) &\leq \sum_{\F_{p^j}}\sum_{m=1}^\infty\sum_{\Irr^\ast(H,\F_{p^j},m)} |G:H|^5p^{(1-|G:H|^{-1})j}\frac{p^{-smj}}{j}\frac{p^{jm}-1}{p^j-1} \\
		&\leq |G:H|^5\sum_{\F_{p^j}}\sum_{m=1}^\infty\sum_{\Irr^\ast(H,\F_{p^j},m)} \frac{p^{(1-|G:H|^{-1}-s)mj}}{j}\frac{p^{jm}-1}{p^j-1} \\
		&= |G:H|^5 \log(\zeta_H)(s-1+|G:H|^{-1}),
	\end{align*}
	which converges.
	
	Conversely, by the same approach, we may define a map $\psi$ from absolutely irreducible $H$-modules $M$ to absolutely irreducible $G$-modules $N$: if $M$ is defined over $\F_{p^j}$, we take $N$ to be an irreducible quotient of $\Ind^G_H(M)$, with the structure of an $\F_{p^{je}}\bra G \ket$-module, where $\F_{p^{je}} \cong \End_G(N)$.
	
	Fix an absolutely irreducible $\F_{p^{\ell}}\bra G \ket$-module $N$ of dimension $n$. Suppose $M \in \psi^{-1}(N)$ is defined over $\F_{p^j}$ (with $j \mid \ell$) and has dimension $m$. As before, we have $|\End_G(N)| \le |\End_H(M)|^{|G:H|}$, so $p^{\ell} \leq p^{j|G:H|}$ and $\ell/j \leq |G:H|$. 
	Denote by $L$ the restriction of scalars from $\F_{p^\ell}$ to $\F_{p^j}$ of $N$. As $L$ is irreducible and $M$ appears as a submodule of $L$, for a fixed $j$ there are at most $|G:H|$ possibilities for $M \in \psi^{-1}(N)$ and so at most $|G:H|^2$ possibilities for $M$ altogether; moreover, such an $N$ must have dimension $n$ at least $(\dim_{\F_{p^j}} L) /{|G:H|}$ over $\F_{p^j}$.
	
	Grouping terms as before, we may write the log of the zeta function of $H$ as $$\sum_{p\in \mathcal{P}} \sum_{\ell=1}^\infty \sum_{n=1}^\infty\sum_{\Irr^\ast(G,\F_{p^{\ell}},n)}\sum_{M \in \psi^{-1}(N)} \alpha(M).$$ Now, for $M$ as above in $\psi^{-1}(N)$,
	\begin{align*}
		\alpha(M) &= \frac{p^{-smj}}{j}\frac{p^{jm}-1}{p^j-1} = 
		\frac{p^{-smj}}{j}\frac{p^{\ell}-1}{p^j-1}\frac{p^{\ell 
				n}-1}{p^{\ell}-1}\frac{p^{jm}-1}{p^{\ell n}-1} \\
		&\leq \frac{p^{-smj}}{j}\frac{\ell p^{\ell-j}}{j}\frac{p^{\ell 
				n}-1}{p^{\ell}-1}\frac{p^{jm}}{p^{\ell n}} = \frac{\ell^2p^{sn 
				\ell-smj+\ell-j+jm-\ell n}}{j^2}\alpha(N).
	\end{align*}
	
	We have $\ell/j \leq |G:H|$, $n \ell/mj \leq |G:H|$ and
	\[
	(s-1)n \ell(1-|G:H|^{-1})+\ell(1-|G:H|^{-1}) -sn\ell \le -|G:H|^{-1} sn\ell,
	\]
	and thus we conclude that $\log(\zeta_H)(s)$ is at most
	\begin{align*}
		&\sum_{p\in \mathcal{P}}\sum_{\ell=1}^{\infty} \sum_{n=1}^\infty\sum_{\Irr^\ast(G,\F_{p^{\ell}},n)} |G:H|^4p^{\ell(1-|G:H|^{-1})((s-1)n+1)}\frac{p^{-sn \ell}}{\ell}\frac{p^{\ell n}-1}{p^{\ell}-1} \\
		&\leq |G:H|^4\sum_{p\in \mathcal{P}}\sum_{\ell=1}^\infty \sum_{n=1}^\infty r^*(G,\F_{p^\ell},n) \frac{p^{-|G:H|^{-1}sn \ell}}{\ell}\frac{p^{\ell n}-1}{p^{\ell}-1} \\
		&= |G:H|^4 \log(\zeta_G)(s/|G:H|),
	\end{align*}
	which converges for $s > |G:H|a(G)$.
	
	For the final statement, if $G = H \times F$, $a(H) \leq a(G)$ by Lemma 
	\ref{quotient}. For the converse, we can argue exactly as in the proof of (i), 
	except that, by \cite[Theorem~2.7]{Fein}, we always have $j=\ell$. We then have $n/m 
	\leq |G:H|$, so from \eqref{eq:ast} we see that, for $s>1$, $\alpha(N) \leq 
	|G:H|\alpha(M)$, and hence $\sum_{N \in \phi^{-1}(M)} \alpha(N) \leq 
	|G:H|^2\alpha(M)$. We conclude as in (i) that $\log(\zeta_G)(s) \leq 
	|G:H|^2\log(\zeta_H)(s)$, which converges for $s > a(H)$.
\end{proof}

\begin{cor}[Finite groups]
	\label{finitegp}
	Let $G$ be a finite group, then $a(G)=1$.
\end{cor}
\begin{proof}	
	Apply Proposition \ref{opensbgp} to $G$ and the trivial open subgroup.
\end{proof}

\begin{rmk}
	\label{perfect}
	The bound in Proposition \ref{opensbgp}(i) can be improved under additional assumptions. For instance, suppose in addition that $H$ is perfect and normal in $G$. Any irreducible representation $N$ of $G$ such that $N \in \phi^{-1}(M)$ with $M$ $1$-dimensional can then be seen as a representation of $G/H$. In particular, $$\sum_{p,j}\sum_{\Irr^\ast(H,\F_{p^j},1)} \sum_{N \in \phi^{-1}(M)} \alpha(N) \leq \log(\zeta_{G/H})(s),$$ which converges for all $s>1$ by Corollary \ref{finitegp}. On the other hand,
	\begin{align*}
		&\sum_{p,j}\sum_{m=2}^\infty\sum_{\Irr^\ast(H,\F_{p^j},m)} |G:H|^5p^{(1-|G:H|^{-1})j}\frac{p^{-smj}}{j}\frac{p^{jm}-1}{p^j-1} \\
		\leq &|G:H|^5\sum_{p,j}\sum_{m=1}^\infty\sum_{\Irr^\ast(H,\F_{p^j},m)} \frac{p^{((1-|G:H|^{-1})/2-s)mj}}{j}\frac{p^{jm}-1}{p^j-1} \\
		= &|G:H|^5 \log(\zeta_H)(s-(1-|G:H|^{-1})/2)
	\end{align*}
	converges when $s> a(H) + (1-|G:H|^{-1})/2$. So we conclude $a(G) \leq a(H) + (1-|G:H|^{-1})/2$. Further variations of this argument are possible.
\end{rmk}

Now we will look at the abscissa of convergence of direct products of 
profinite groups. We start by considering the case of profinite rings.

\begin{lem}[Products of rings]
	For PFG profinite rings $R,S$, $a(R \times S) = \max(a(R),a(S))$.
\end{lem}
\begin{proof}
	This is clear from Proposition \ref{prop:products}.
\end{proof}

Unfortunately, as we have seen in Section~\ref{sec:zeta_def}, the zeta function 
of the product of two groups is not as well-behaved. Nonetheless, we can 
produce an upper bound.

\begin{prop}[Products]
	\label{abscissaproducts}
	$a(G \times H) \leq a(G) + a(H)$.
\end{prop}
\begin{proof}
	By \cite[Theorem 2.7]{Fein}, absolutely irreducible $\F_{p^j}\bra G \times H \ket$-modules have the form $M \otimes_{\F_{p^j}} N$, with $M$ an absolutely irreducible $\F_{p^j}\bra G \ket$-module and $N$ an absolutely irreducible $\F_{p^j}\bra H \ket$-module. Fix such an $M$, of dimension $m$, say. By the contribution of $M$ to the sum $\log(\zeta_{G \times H})(s)$, we mean the sum, over all $\F_{p^j}\bra G \times H \ket$-modules of the form $M \otimes_{\F_{p^j}} N$, of $\frac{1}{j}p^{-smkj} |\mathbb{P}^{mk-1}(\F_{p^j})|$, where $k$ is the dimension of $N$.
	
	When $s = a(H) + \varepsilon + t$, with $\varepsilon>0$, the contribution of $M$ is therefore $$\sum_{k=1}^\infty \frac{r^\ast(H,\F_{p^j},k)}{j} p^{-(a(H)+\varepsilon)kj} \frac{p^{kj}-1}{p^j -1} p^{-(sm-a(H)-\varepsilon)kj} \frac{p^{mkj}-1}{p^{kj}-1}.$$ Now
	\begin{align*}
		p^{-(sm-a(H)-\varepsilon)kj} \frac{p^{mkj}-1}{p^{kj}-1} 
		&\leq \frac{p^{kj(m-1-sm+a(H)+\varepsilon)}}{1-p^{-kj}} \leq 
		c_\epsilon\frac{p^{j(m-1-sm+a(H)+\varepsilon)}}{1-p^{-j}}
	\end{align*}
	for a constant $c_\epsilon$ depending only on $\epsilon$, because $$m-1-sm+a(H)+\varepsilon = -tm - (a(H)+\varepsilon-1)(m-1) < -tm < 0.$$ So the contribution of $M$ is bounded above by
	\begin{align*}
		c_\epsilon\log(\zeta_H)(a(H)+\varepsilon) \frac{p^{-tmj}}{1-p^{-j}}
		&\leq 2c_\epsilon \log(\zeta_H)(a(H)+\varepsilon) p^{-tmj} |\mathbb{P}^{m-1}(\F_{p^j})|.
	\end{align*}
	Hence, when $t = a(G)+\delta$, 
	$\delta>0$, the total sum over all $M$ is $$\leq 
	2c_\epsilon\log(\zeta_H)(a(H)+\varepsilon) \sum_{p \in \mathcal{P}} \sum_{j=1}^\infty 
	\sum_{m=1}^\infty r^\ast(G,\F_{p^j}, m) p^{-(a(G)+\delta)mj} 
	|\mathbb{P}^{m-1}(\F_{p^j})|,$$ which converges by Lemma \ref{sameabscissa}. 
	This holds for all $\varepsilon,\delta>0$, so the abscissa of convergence is at 
	most $a(G)+a(H)$.
\end{proof}

Recall from \cite[Theorem 4.7]{CCKV} that split extensions of profinite groups with UBERG have UBERG. We cannot get the same bound on the abscissa for such groups as we do for direct products, but we can get close.

\begin{thm}[Split extensions]\label{thm:split}
	Suppose $G$ is a profinite group, $K \unlhd G$ and the quotient map $G \to G/K$ splits. Then $a(G) \leq a(K) + a(G/K)+1$.
\end{thm}

We will use the following elementary  lemma. 
\begin{lem}\label{lem:inequality}
	For all $a\in \mathbb{N}$, $p\ge 2$ and $t>0$, there is a constant $c_0$, 
	depending only on $t$, 
	such that $$\sum_{l \geq a} l^2p^{-tl} \leq c_0a^2 \cdot p^{-at}.$$
\end{lem}
\begin{proof}
Divide by $a^2p^{-at}$ and observe that
\[
	\sum_{l = a}^\infty \Bigl(\frac{l}{a}\Bigr)^2 p^{-t(l-a)} \leq \sum_{l=0}^\infty (l+1)^2 p^{-tl} \leq  \sum_{l= 0}^\infty (l+1)^2 2^{-tl}
\]
converges to a value which does not depend on $p$ and $a$.
\end{proof}

To simplify the notation, sums in this proof will mean sums from $1$ to $\infty$ unless stated otherwise.

\begin{proof}[Proof of Theorem~\ref{thm:split}]
	For each absolutely irreducible $\F_{p^j}\bra G \ket$-module $W$ of dimension $n$, fix an irreducible summand $V$ of $\Res^G_K N$. Then $M$ is absolutely irreducible over $\F_{p^k}$, some multiple $k$ of $j$, of dimension $m$ such that $mk \leq nj$. This gives a map $\phi$ from the set of absolutely irreducible $G$-modules to the set of absolutely irreducible $K$-modules, and we write $r^\ast(G,\F_{p^j},n,V)$ for the number of absolutely irreducible $\F_{p^j}\bra G \ket$-modules $W$ of dimension $n$ such that $\phi(W)=V$. Similarly, if $N$ is an irreducible $\F_p\bra G \ket$-module, we may think of $N$ as an absolutely irreducible $\End_G(N)\bra G \ket$-module $N'$, and define a map $N \mapsto \phi(N')$ from irreducible $G$-modules over prime fields to absolutely irreducible $K$-modules. Then we write $R(G,\F_p,n,M)$ for the number irreducible $\F_p\bra G \ket$-modules of dimension $n$ in the preimage of $M$.
	
	We call the contribution of $V$ to $\log(\zeta_G)(s)$ the sum $\sum_{\phi^{-1}(V)} \alpha(W)$, where $\alpha(W) = 
	\frac{1}{j} p^{-snj}\frac{p^{nj}-1}{p^j-1}$. Thus, the contribution of $V$ is 
	\[
	\sum_{n,j} \frac{r^\ast(G,\F_{p^j},n,V)}{j}p^{-snj}\frac{p^{jn}-1}{p^j-1}.
	\]
	Now $r^\ast(G,\F_{p^j},n,M) \leq R(G,\F_{p^j},n,M),$ and the proof of \cite[Theorem 4.7]{CCKV} shows that $R(G,\F_{p^j},n,M) \leq n R(G/K,\F_{p^j},n);$ finally, $$R(G/K,\F_{p^j},n) \leq \sum_{l=1}^{n}\sum_{i|l} r^\ast(G/K,\F_{p^{ji}},l/i).$$
		
	Let $s' = a(K)+\varepsilon+s$, for some small $\varepsilon>0$. Then the contribution of $V$ to $\log(\zeta_G)(s')$ is
	\begin{align*}
		&\leq \sum_{n,j} \frac{r^\ast(G,\F_{p^j},n,V)}{j}p^{-s'nj}\frac{p^{jn}-1}{p^j-1} \\
		&\leq p^{-(a(K)+\varepsilon)mk}\sum_{n,j}\sum_{u,v:v \geq j, uv \leq nj} \frac{n}{j} r^\ast(G/K,\F_{p^v},u)p^{-snj}\frac{p^{jn}-1}{p^j-1} \\
		&= p^{-(a(K)+\varepsilon)mk}\sum_{u,v} r^\ast(G/K,\F_{p^v},u) \sum_{n,j: j \leq v, nj \geq uv} \frac{n}{j} p^{-snj}\frac{p^{jn}-1}{p^j-1}.
	\end{align*}
	Now grouping the terms with the same value of $nj$ together, we get
	\begin{align*}
		\sum_{n,j: nj \geq uv} \frac{n}{j} p^{-snj}\frac{p^{jn}-1}{p^j-1} \leq \sum_{n,j: nj \geq uv} njp^{(1-s)nj} \leq
        \sum_{l \geq uv} l^2p^{(1-s)l} \leq (c_0u^2v^2) p^{vu(1-s)}
	\end{align*}
	for some constant $c_0$ depending only on $s$, by Lemma~\ref{lem:inequality} for $a=uv$ and $t=1-s$. 
	
	So the contribution of $V$ is
	\begin{align*}
		&\leq p^{-(a(K)+\varepsilon)mk}\sum_{u,v} c_0u^2v^2r^\ast(G/K,\F_{p^v},u)p^{vu(1-s)} \\
		&\leq p^{-(a(K)+\varepsilon)mk}\sum_{u,v} c_0u^2v^3 \frac{r^\ast(G/K,\F_{p^v},u)}{v} p^{vu(1-s)} |\mathbb{P}^{u-1}(\F_{p^v})|
	\end{align*}
	which converges, by Lemma \ref{sameabscissa}, for all $s > a(G/K)+1$, to $p^{-(a(K)+\varepsilon)mk}f(s) \ll \frac{1}{k} p^{-(a(K)+\varepsilon/2)mk} |\mathbb{P}^{m-1}(\F_{p^k})|f(s)$ for some $f$ which is independent of $V$. Summing over all absolutely irreducible $K$-modules $V$, we see that $\log(\zeta_G)(s')$ converges for all $s' > a(K)+a(G/K)+1$.
\end{proof}

\section{Examples of abscissae}\label{sec:examples}

In this section we will prove Theorem~\ref{thmABC:examples}.

\subsection{Free abelian groups}\label{sec:free-abelian-abscissa}

Finitely generated abelian groups have UBERG by \cite{KV}, and all their absolutely irreducible representations have dimension $1$. For the free abelian group $\Z^r$ of rank $r$ and its profinite completion $\hat{\Z}^r$, $r^\ast(\Z^r,\F_{p^j},1)= r^\ast(\hat{\Z}^r,\F_{p^j},1)$ is the number of homomorphisms from $\hat{\Z}^r$ to $\F_{p^j}^\times$, that is, $(p^j-1)^r$. So $$\log(\zeta_{\hat{\Z}^r}(s)) = \sum_{p \in \mathcal{P}} \sum_{j=1}^\infty \frac{(p^j-1)^r}{j}p^{-sj} \leq \sum_n n^{r-s}$$ which converges for $s > r+1$. When $s=r+1$, we get $$\log(\zeta_{\hat{\Z}^r}(s)) > \sum_{p \in \mathcal{P}} \frac{(p-1)^r}{p^{r+1}} = \sum_{p \in \mathcal{P}} \sum_{i=0}^r \frac{(-1)^i}{p^{i+1}} \binom{r}{i}.$$ For each $i > 0$, the sum over the primes $p$ converges absolutely, whereas for $i = 0$, it diverges \cite{mertens}. Therefore the whole sum, over $i$ and $p$, diverges. So $a(\hat{\Z}^r)=r+1$.

Note that, by expanding $(p^j-1)^r = \sum_{i=0}^r (-1)^{r-i} \binom{r}{i} p^{ij}$, we get $$\log(\zeta_{\hat{\Z}^r}(s)) = \sum_{i=0}^r \sum_{p \in \mathcal{P}} \sum_{j=1}^\infty (-1)^{r-i} \binom{r}{i} \frac{p^{(i-s)j}}{j} = \sum_{i=0}^r (-1)^{r-i} \binom{r}{i} \log(\zeta(s-i)),$$ so $\zeta_{\hat{\Z}^r}(s) = \prod_{i=0}^r \zeta(s-i)^{(-1)^{r-i} \binom{r}{i}}$, where $\zeta(s)$ is the Riemann zeta function. 
In particular, $\zeta_{\hat{\Z}^r}(s)$ admits a meromorphic extension to $\C$ and has a simple pole at $r+1$.
 It is interesting to remark that our zeta function in this case is very similar to the subgroup growth zeta function as defined in \cite{GSS}: in that case $\zeta_{\hat{\mathbb{Z}}^r}^\le(s)=\prod_{i=1}^{r-1} \zeta(s-i)$ by \cite[Proposition 1.1]{GSS}.

\subsection{Free abelian pro-\texorpdfstring{$p$}{p} 
groups}\label{sec:free_ab_pro_p}

Fix a prime number $p$. Here we study the abscissa of convergence of free abelian pro-$p$ groups $\Z_p^r$. 
Let $w_p(m) = |m|_p^{-1}$ denote the highest $p$-power dividing $m$. If $q$ is a prime power, then 
\[
	r^*(\Z_p^r,\F_q,1) = w_p(q-1)^r,
\]
i.e., there are many absolutely irreducible representations, if $q$ is congruent $1$ modulo a high power of $p$.
We will see that this observation links the abscissa closely to small prime powers in the arithmetic progressions $np^k + 1$. 

We begin with a short summary of results on small \emph{primes} in arithmetic progressions.
 By Linnik's theorem there are constants $c,L$
such that for every $d \geq 2$ the least prime $p_{\text{min}}(d)$ congruent $1$ modulo $d$ satisfies
\[
	p_{\text{min}}(d) \leq c d^L.
\]
Currently, the best known value for the exponent is $L=5$; see \cite{Xylouris}.
Assuming the extended Riemann hypothesis or the generalized Riemann hypothesis, we have $L = 2+\epsilon$ for every $\epsilon > 0$; see \cite{BachSorenson} or \cite{HB}. 
A folklore conjecture (sometimes attributed to Chowla) states that $L = 1+\epsilon$ for every $\epsilon>0$.

For our purposes the only relevant case is $d = p^j$ a power of the fixed prime number $p$. In this case better results are known.
Let $L(p)$ be defined as
\[
 L(p) = \limsup_{j \to \infty} \frac{\log(p_{\text{min}}(p^j))}{j \log(p)}.
\]
In other words $L(p)$ is the infimum over all real numbers $L> 0$ such that $p_{\text{min}}(p^j) \leq c_L p^{jL}$ for some $c_L > 0$ and all $j \geq 1$.
Barban, Linnik and Tshudakov proved $L(p) \leq \frac{8}{3}$. Gallagher \cite{Gallagher} (see also Iwaniec \cite{Iwaniec}) established $L(p) < 2.5$ and Huxley \cite{Huxley} improved this to $L(p) \leq 2.4$. Currently the best bound appears in a paper of Banks-Shparlinski \cite{BanksShparlinski}, who show that $L(p) < 2.1115$.

We will now see that the abscissa of convergence for $\Z_p^r$ is related to a very similar, but less studied constant. For our purposes, we can replace $\limsup$ by $\liminf$ and, in addition, we are interested in \emph{prime powers} in arithmetic progressions (which should not make a big difference asymptotically).
\begin{defn}
Let $pp_{\text{min}}(d)$ denote the least prime power congruent $1$ modulo $d$. We define
\[
	K(p) = \liminf_{j \to \infty} \frac{\log(pp_{\text{min}}(p^j))}{j \log(p)}.
\]
\end{defn}
\begin{prop}\label{prop:abscissa-Zp}
Let $p$ be a prime number. Then
\[
    \frac{r}{K(p)} \leq a(\Z_p^r) \leq \frac{r-1}{K(p)} + 1.
\]
In particular, $a(\Z_p) = 1$. 
\end{prop}
We note that $a(\Z_p) \geq 1$ for all groups, so the second assertion follows 
immediately from the upper bound.

It is clear that $1 \leq K(p) \leq L(p)$. As mentioned before, it is 
conjectured that $L(p) = 1$ and so one might conjecture $K(p) = 1$. In this 
case the upper and lower bounds agree and $a(\Z_p^r) = r$. 

\begin{proof}[Proof of the lower bound in \ref{prop:abscissa-Zp}]
Let $q^{k_j}_j = pp_{\text{min}}(p^j)$ be the least prime power congruent $1$ modulo $p^j$ (where $q_j$ is a prime).

Let $\varepsilon > 0$. By assumption, $q_j^{k_j} > p^{j(K(p)-\varepsilon)}$ for all sufficiently large $j$, say $j \geq j_0$. Since $q_j^{k_j} \leq cp^{5j}$ by \cite{Xylouris}, we have 
$k_j \log(q_j) \leq \log(c)+5j\log(p)$.
Then, for all real $s> 1$, $\log\zeta_{\Z_p^r}(s)$ is at least
\[
	 \sum_{j\geq j_0} \frac{p^{jr}}{k_j q_j^{k_js}} \geq  
	\sum_{j \geq j_0} \frac{p^{jr}}{k_j p^{js(K(p)-\varepsilon)}}
	\geq \sum_{j \geq j_0} \frac{p^{jr}\log(q_j)}{(\log(c)+5j \log(p)) 
p^{js(K(p)-\varepsilon)}}
\]
and this series diverges for $s < \frac{r}{K(p)-\varepsilon}$.
We deduce that $\alpha(\Z_p^r) \geq \frac{r}{K(p)}$. 
\end{proof}
Let $\zeta_{\Z^r_p}$ be the representation zeta function of $\Z_p^r$. Then
\[
	\log \zeta_{\Z_p^r}(s) = \sum_{q \text{ prime}} \sum_{k=1}^\infty \frac{w_p(q^k-1)^r}{k} q^{-sk};
\]
recall that $w_p(m)$ denotes the highest $p$-power dividing $m$ with the convention that $w_p(0)= 0$.
It is well-known that the Riemann zeta function satisfies
\[
	\log \zeta(s) = \sum_{n} \frac{\Lambda(n)}{\log(n)} n^{-s}
\]
where $\Lambda$ denotes the von Mangoldt function. In the same way we can rewrite $\zeta_{\Z^r_p}(s)$ and obtain
\[
	\log \zeta_{\Z_p^r}(s) = \sum_{n}  \frac{\Lambda(n) w_p(n-1)^r}{\log(n)} n^{-s}.
\]

\begin{proof}[Proof of the upper bound in Proposition \ref{prop:abscissa-Zp}]
Let $K := K(p)$.
Let $\varepsilon > 0$ be given. We may assume that $K - \varepsilon \geq 1$, since 
for $K = 1$, we can even take $\varepsilon = 0$. 

By assumption, there are only finitely many pairs $(j,n)$ where $n$ is a prime power with $n \equiv 1 \bmod p^j$ and $n^{1/(K-\varepsilon)} \leq p^j$.
We define $w^\varepsilon_p(n-1)$ to be the highest power of $p$ which divides $n-1$ and is at most $n^{1/(K-\varepsilon)}$.
Using $\Lambda(n)/\log(n) \leq 1$ we obtain for all real $s > 1$:
\begin{align*}
\log \zeta_{\Z_p^r}(s) &= \sum_{n}  \frac{\Lambda(n) w_p(n-1)^r}{\log(n)} n^{-s}\\
	 &=  \sum_{n}  \frac{\Lambda(n) w^{\varepsilon}_p(n-1)^r}{\log(n)} n^{-s} + O(1)\\	
	 &\leq \sum_n w^\varepsilon_p(n-1)^r n^{-s} + O(1).
\end{align*}
	
Let $W^\varepsilon_{p,r}(x) = \sum_{n\leq x} w^\varepsilon_p(n-1)^r$.
Using a standard trick we get
\begin{align*}
		W^\varepsilon_{p,r}(x) &= \sum_{n \leq x} w^\varepsilon_p(n-1)^r  \leq \lfloor x-1\rfloor +\sum_{k=1}^{\lfloor\log_p(x)/(K-\epsilon)\rfloor}  \Bigl\lfloor\frac{x-1}{p^k}\Bigr\rfloor (p^{kr}-p^{(k-1)r}) \\
				&\leq (x-1) + \sum_{k=1}^{\lfloor\log_p(x)/(K-\epsilon)\rfloor} \frac{x-1}{p^k}(p^{rk}-p^{r(k-1)}) \\
				&= (x-1) + (1-p^{-r})(x-1)\sum_{k=1}^{\lfloor\log_p(x)/(K-\epsilon)\rfloor}  p^{(r-1)k} \\
				&\leq (x-1)\left(1+\frac{1-p^{-r}}{1-p^{-(r-1)}}x^{(r-1)/(K-\epsilon)}\right)\\
				&\ll x^{(r-1)/(K-\epsilon)+1}.
			\end{align*}
We use Abel's summation formula to obtain for all $x > p$
\begin{align*}
	\sum_{p < n \leq x} w^\varepsilon_p(n-1)^r n^{-s} &= W^{\varepsilon}_{p,r}(x)x^{-s} - W^{\varepsilon}_{p,r}(p)p^{-s} - \int_{p}^x W^{\varepsilon}_{p,r}(u) (-su^{-s-1}) d u\\
	&\ll x^{\frac{r-1}{K-\epsilon}+1-s} + s\int_{p}^x u^{(r-1)/(K-\epsilon)-s} du\\
	&\ll x^{\frac{r-1}{K-\epsilon}+1-s} + \frac{s}{\frac{r-1}{K-\epsilon}+1-s}x^{\frac{r-1}{K-\epsilon}+1-s}
\end{align*}
the latter expression is bounded for all $s>\frac{r-1}{K-\epsilon}+1$ as $x \to \infty$.
We conclude that $\log \zeta_{\Z_p^r}(s)$ converges absolutely for $\Re(s) > \frac{r-1}{K-\epsilon}+1$.
\end{proof}
\begin{rmk}
Assuming $K(p) > 1$, it seems that the upper bound in Proposition \ref{prop:abscissa-Zp} is of the right order of magnitude. In fact, from the prime number theorem in arithmetic progressions one would expect that there are 
roughly 
\[
	\frac{p^{j(K(p)+\varepsilon)}}{\varphi(p^j)} = \frac{p^{j(K(p)+\varepsilon-1)}}{1-p^{-1}}
\]
 primes congruent $1$ mod $p^j$ below $p^{j(K(p)+\varepsilon)}$ (at least for large $j$). 
 An estimate of this form gives rise to a lower bound which matches the upper bound in Proposition \ref{prop:abscissa-Zp}. 
 
Some quantitative results about the amount of small primes in arithmetic progressions are known.
 For instance, Xylouris \cite[Theorem 2.2]{Xylouris} proved that for large $p$ there are 
 at least $p^{j(3.21)}$ primes of the form $np^j+1$ below $p^{5j}$. This imposes the lower bound
 \[
 	a(\Z_p^r) \geq \frac{3.21 + r}{5}.
 \]
 For $r=2$ this improves on the lower bound in Proposition 
\ref{prop:abscissa-Zp} and gives $a(\Z_p^2) > 1.04 > 1$.
 \end{rmk}

\begin{rmk}
	The proof of the upper bound shows similarly that any finite product $\prod \Z_p$ over distinct primes $p$ has abscissa of convergence $1$ (though, as we have already seen, the product over all primes has abscissa of convergence $2$). But in fact, we can show that in some cases, $\prod_k \Z_{p_k}$ has abscissa of convergence $1$ for an infinite sequence of primes $(p_k)$.
	
	Indeed, suppose we choose $p_k > 2^{(2^k)}$, and let $G = \prod_k \Z_{p_k}$. We may define $c_{(p_k)}(n)$ to be the largest divisor of $n$ all of whose prime factors are in $(p_k)$, so that $\log(\zeta_G)(s) \leq \sum_n c_{(p_k)}(n-1)n^{-s}$. Write $f(x)$ for the number of positive integers $n$ less than $x$ such that $c_{(p_k)}(n) = n$. As in Proposition \ref{prop:abscissa-Zp}, Tschebyscheff's trick shows that $\sum_{n \leq x} c_{(p_k)}(n-1)$ is at most $(x-1)f(x)$.
	
	Let $x = 2^{(2^k)}$; then $f(x)$ is at most the number of partitions of $2^k$ into powers of $2$. It is shown in \cite[(1.3)]{deBruijn} that the number of such partitions is $2^{O(k^2)}$. Therefore, $\sum_{n \leq 2^{(2^k)}} c_{(p_k)}(n-1) \leq 2^{2^k+O(k^2)}$. Finally, we can use Abel's summation formula as before to show that $\sum_{n \leq 2^{(2^k)}} c_{(p_k)}(n-1)n^{-s} = O(2^{2^k+O(k^2)-2^ks})$ converges when $s>1$.
\end{rmk}

\subsection{Free pro-\texorpdfstring{$\mathfrak{C}$}{C} groups, I}\label{sec:free-pro-c}

Let $\mathfrak{C}$ be a class of finite groups which is closed under quotients, finite subdirect products, taking normal subgroups and extensions (i.e., $\mathfrak{C}$ an NE-formation in the sense of \cite[\S 2.1]{RZ}).
Let $F^\mathfrak{C}_r$ be the free pro-$\mathfrak{C}$ on $r$ generators; these exist by \cite[Section 3.3]{RZ} and satisfy the usual universal property.
An open normal subgroup of index $m$ is again a free pro-$\mathfrak{C}$ group of rank $m(r-1)+1$ by \cite[Theorem 3.6.2]{RZ}.
Here we prove a general lower bound result for the abscissa of $F_r^\mathfrak{C}$.
\begin{thm}\label{thm:lowerbound-general}
Assume that $\mathfrak{C}$ contains a non-trivial abelian group and an absolutely irreducible subgroup $S \subseteq \GL(n_0,q)$ for some $n_0 \geq 1$ and some prime power $q$.
Then
\[
	a(F_r^{\mathfrak{C}}) \geq \frac{\log |S|}{n_0 \log q}(r-1) +1.
\]
\end{thm}

\begin{lem}\label{lem:invariants-extend}
Let $p$ be a prime and let $L/F$ be an extension of finite fields such that $p$ divides $(|L|-1)/(|F|-1)$. Let $G$ be a profinite group and let $N \leq G$ be a normal subgroup of index $p$.
Every $G$-invariant absolutely irreducible representation of $N$ over $F$ 
extends to an absolutely irreducible representation of $G$ over $L$.
\end{lem}
\begin{proof}
Let $\theta$ be an absolutely irreducible $G$-invariant representation of $N$ defined over $F$.
Let $\omega \in H^2(G/N, F^\times)$ be the obstruction cocycle (see \cite[Section 9]{Karpilovsky}).
Since $G/N$ is cyclic of order $p$, we have $H^2(G/N,A) = A/pA$ for all modules $A$ with trivial $G/N$-action (see \cite[Lemma 5.5]{Karpilovsky}).
Since $F^\times$ and $L^\times$ are cyclic the map  $H^2(G/N,F^\times) \to H^2(G/N, L^\times)$ is trivial, i.e., by \cite[Thm.~9.6]{Karpilovsky} $\theta$ extends to $G$ over $L$.
\end{proof}

\begin{proof}[Proof of Theorem \ref{thm:lowerbound-general}]
We assume $r\geq 2$; in fact, for $r = 1$ there is nothing to do. We may also assume that $S$ is non-trivial.
Since $\mathfrak{C}$ contains a non-trivial abelian group, there is a prime number $p$ such that $\mathfrak{C}$ contains all $p$-groups.
We only consider representations over the field $\F_q$. Let $m = |S|$.
The number of surjective homomorphisms from $F_r^{\mathfrak{C}}$ onto $S$ is $m^r - O((m/2)^r)$, where the error term depends on the maximal subgroups of $S$. 
A conjugacy class of homomorphisms contains at most $\frac{|\GL(n_0,q)|}{q-1}$ elements. 
We deduce that the number of equivalence classes of representations $r^S(F_r^{\mathfrak{C}})$ with image $S \subseteq \GL(n_0,q)$ satisfies
\[
	r^S(F_r^{\mathfrak{C}}) \geq \lambda m^{r-1}\bigl(1- O(2^{-r+1})\bigr)
\]
as $r$ tends to infinity with $\lambda = \frac{m(q-1)}{|\GL(n_0,q)|}$.

Let $k \geq 1$ and let
$\phi \colon F_r^{\mathfrak{C}} \to \Z/p^{k}\Z$ be a surjective homomorphism. Let $N = \ker(\phi)$ and let $N^+$ be the unique normal subgroup containing $N$ with index $p$. We note that $N$ is free of rank $p^k(r-1)+1$ and $N^+$ is free of rank $d = p^{k-1}(r-1)+1$.
We claim that most irreducible representations of $N$ with image $S \subseteq \GL(n_0,q)$ are not $N^+$-invariant. 
From Lemma \ref{lem:invariants-extend} we know that an $N^+$-invariant representation of $N$ over $\F_q$ extends to $N^+$ over a suitable field extension $L/\F_q$ (whose degree depends only on $p$ and $q$).
Let $x_1,\dots, x_d$ be a free generating set of $N^+$ with $x_1,\dots, x_{d-1} \in N$ (To find such a generating set, one can start with an arbitrary free generating set of $N^+$ such that $\phi(x_d)$ has order $p^k$ and replace $x_i$ by $x_ix_d^{k_i}$ for suitable $k_i \in \Z$). The number of homomorphisms of $N^+$ into $\GL_{n_0}(L)$ that map $x_1,\dots, x_{d-1}$ into $S \subseteq \GL(n_0,q)$ is $m^{d-1} \cdot |\GL_{n_0}(L)|$
and we conclude that the number of $N^+$-invariant representations satisfies
\begin{align*}
	r^S(N)^{N^+} = O(m^{p^{k-1}(r-1)}).
\end{align*}
and so
\begin{align*}
   r^S(N) -  r^S(N)^{N^+} \geq \lambda m^{p^k(r-1)}\left(1- O\bigl(2^{-p^{k-1}(r-1)}\bigr)\right).
\end{align*}
If $\theta$ is an absolutely irreducible representation of $N$ which is not $N^+$-invariant, then $\theta$ has trivial inertia subgroup in $F_r^{\mathfrak{C}}$ since $N^+$ is minimal normal in $G/N$. Therefore the induced representation $\Ind_N^{F_r^{\mathfrak{C}}}(\theta)$ is absolutely irreducible.
Moreover, only $|G/N| = p^k$ distinct conjugates of $\theta$ give rise to the same induced representation.
We obtain that
\begin{align*}
   r^*(F_r^{\mathfrak{C}},\F_q,p^k {n_0}) \geq \kappa p^{-k} m^{p^k(r-1)}
\end{align*}
for all sufficiently large $k$ and a suitable constant $\kappa$.
This lower bound implies that the series 
\[\log(\zeta_{F_r^{\mathfrak{C}}})(s) = \sum_{n,\ell,j} \frac{r^*(F_r^{\mathfrak{C}},\F_{\ell^j},n)}{j} \ell^{-jn s } \frac{\ell^{jn}-1}{\ell^{j}-1} 
\geq  \kappa \sum_{k} m^{p^k(r-1)}p^{-k}q^{(1-s)p^{k}{n_0}} 
\] 
diverges for real numbers $s <  \frac{\log(m)}{{n_0}\log(q)}(r-1) + 1$.
\end{proof}
\begin{cor}\label{cor:lowerbound-general-perm}
In Theorem~\ref{thm:lowerbound-general} assume in addition that $\mathfrak{C}$ contains a transitive permutation group $T$ of degree $d > 1$.
Then
\[
	a(F_r^{\mathfrak{C}}) \geq \frac{\log\bigl( |S|^{(d-1)}|T|\bigr)}{(d-1)n_0 \log q}(r-1) +1.
\]
\end{cor}
\begin{proof}
Let $k \geq 1$. We consider the wreath product
\[
	 W_k = S\wr \underbrace{T \wr T \wr \cdots\wr T}_{k \text{ times }}
\]
constructed from the permutation representation of $T$ on $d$ elements.
Since $W_k$ is an extension of direct products of groups in $\mathfrak{C}$, it belongs to $\mathfrak{C}$.
We will show that $W_k$ has a faithful absolutely irreducible representation of degree $d^{k}n_0$ over $\F_q$.
Then Theorem \ref{thm:lowerbound-general} gives
\begin{align*}
	a(F_r^{\mathfrak{C}})-1 &\geq \frac{\log|W_k|}{d^kn_0\log(q)}(r-1) = \frac{\log(|S|^{d^k}|T|^{(d^k-1)/(d-1)})}{d^kn_0\log(q)}(r-1) \\ &= \frac{\log(|S|^{d-1})+(1-\frac{1}{d^k})\log(|T|))}{(d-1)n_0\log(q)}(r-1)
\end{align*}
and the result follows by letting $k$ tend to infinity.

Let $N = \prod_{i=1}^{d^k}S_i$ be the normal base group of $W_k$. To see that $W_k$ admits an absolutely irreducible faithful representation of degree $d^kn_0$ we fix a copy $S_1 \subseteq N$ of $S$ in the base group. Then $S_1$ has a faithful representation $\theta$ of degree $n_0$ over $\F_q$. We extend this representation trivially to a representation $\theta'$ of the normalizer $N_{W_k}(S_1)$. This is possible since the normalizer is a direct product $S_1 \times W'$; indeed it is generated by $N$ and a point stabilizer in $T\wr T \wr \dots \wr T$. In particular, the normalizer has index $d^k$ in $W_k$. The induced representation  $\Ind_{N_{W_k}(S_1)}^{W_k}(\theta')$ has degree $d^kn_0$. It is absolutely irreducible because the restriction to $N$ consists of a single $W_k$-orbit of irreducible representations. Moreover, it is faithful, because both $\theta$ and the permutation representation of $T\wr T \wr \dots \wr T$ on $d^k$ elements are faithful.
\end{proof}

We also include here a result which will be useful in proving upper bounds.

\begin{prop}
	\label{upperbound-images}
	Let $G \leq GL(n,p^j)$. Then the number of absolutely irreducible representations of dimension $n$ over $\F_{p^j}$ of an $r$-generated group $F$ with image contained in $G$ is at most $|G|^{r-1}|G \cap \F_{p^j}^\times| \leq |G|^{r-1}p^j$, where $\F_{p^j}^\times \leq GL(n,p^j)$ is the group of diagonal matrices.
\end{prop}
\begin{proof}
	An absolutely irreducible representation of $F$ is just a homomorphism $F \twoheadrightarrow J \leq GL(n,p^j)$, up to conjugacy in $GL(n,p^j)$, with $J$ absolutely irreducible.  The number of such homomorphisms with $J \leq G$ is at most $|G|^r$. The size of a conjugacy class is at least $|N_{GL(n,p^j)}(G)|/|C_{GL(n,p^j)}(J)| \geq |G|/|C_{GL(n,p^j)}(J) \cap G| \geq |G|/p^j$, because $J$ is absolutely irreducible. Therefore the number of conjugacy classes is at most $|G|^r \frac{p^j}{|G|}$, as required.
\end{proof}

\subsection{Free pro-\texorpdfstring{$p$}{p} groups (odd \texorpdfstring{$p$}{p})}\label{sec:free-pro-p}
Let $p$ be an odd prime number.
We define 
\[K'(p) = \inf_{k \geq 1} \frac{\log(pp_{\min}(p^k))}{(k+\frac{1}{p-1})\log(p)}.\] 

\begin{thm}\label{thm:abscissa-free-pro-p}
Let $F^p_r$ be the free pro-$p$ group on $r$ generators with $r\geq2$. The abscissa of convergence is
\[
	a(F^p_r) = \frac{r-1}{K'(p)} +1.
\]
\end{thm}
Before we give the proof, some comments on the constant $K'(p)$ are in order. We note that
\[
	K'(p) \leq \inf_{k \geq 1} \frac{\log(pp_{\min}(p^k))}{k\log(p)} \leq K(p),
\]
where $K(p)$ is the constant from Section~\ref{sec:free_ab_pro_p}. The discussion there gives $K'(p) \leq 2.1115$ and conjecturally $K'(p) \leq 1$. 
One can determine the precise value of $K'(p)$ if $p$ is a Mersenne prime.
\begin{prop}
	\label{Mersenne}
	Let $p$ be an odd prime.
	\begin{enumerate}[(i)]
		\item If $p$ is not a Mersenne prime, $K'(p) \geq 1$.
		\item If $p$ is a Mersenne prime, $K'(p) = \frac{(p-1)\log(p+1)}{p\log(p)} < 1$.
	\end{enumerate}
\end{prop}
\begin{proof}
	\begin{enumerate}[(i)]
		\item Since the size $S(n,q^j)$ of a Sylow $p$-subgroup of $GL(n,q^j)$ is at most $w_p(q^j-1)^np^{\frac{n-1}{p-1}}$, with equality when $n$ is a power of $p$, we see that $K'(p)$ is maximal such that $|S(n,q^j)|^{K'(p)} \leq q^{nj}$. So when $p$ is not a Mersenne prime, $K'(p) \geq 1$ by \cite[Theorem 1.6(ii)]{Wolf}; the result follows.
		\item 
		Suppose $p = 2^m -1$ is a Mersenne prime. Let $f_p(k)=\frac{\log(p^k+1)}{(k+1/(p-1))\log(p)} \leq \frac{\log(pp_{\min}(p^k))}{(k+1/(p-1))\log(p)}$. When $p>3$, a calculation shows $f_p(k)$ has a minimum when $k=1$, and so $K'(p) = f_p(1)$. When $p=3$, $f_3(k)>f_3(1)$ for all $k \geq 3$. Since $19$ is the smallest prime power congruent $1$ modulo $9$, we conclude that for $k=2$ we have $\frac{\log(19)}{(2+1/(p-1))\log(p)}>f_3(1)$, so we conclude that $K'(3) = f_3(1)$.
	\end{enumerate}
\end{proof}

\begin{ex}\label{ex:values}
For the first four Mersenne primes  one obtains
\[
	K'(3)\approx 0.8412,\quad K'(7)  \approx 0.9160,\quad K'(31) \approx 0.9767,\quad K'(127) \approx 0.9937.
\]
Even if $p$ is not a Mersenne prime the formula for $K'(p)$ yields upper bounds close to $1$ already for small values of $k$. For instance,
\[
	K'(5) \leq \frac{\log(1251)}{(4+\frac{1}{4})\log(5)} \approx 1.0426.
\]
\end{ex}

\begin{ex}\label{ex:free_prop}
	Let $G = F^p_r$. Then $a(G) = \frac{r-1}{K'(p)} + 1$. If $H$ is an open subgroup of index $i$, by the Nielsen-Schreier theorem for free pro-$p$ groups \cite[Theorem 3.6.2]{RZ}, $H$ is free pro-$p$ on $i(r-1)+1$ generators; hence $a(H) = \frac{i(r-1)}{K'(p)}+1$. This gives $a(H)/a(G) = i \frac{1+\frac{K'(p)}{i(r-1)}}{1+\frac{K'(p)}{(r-1)}}$. Since the right hand side approaches $i$ as $r$ tends to infinity, this shows that the upper bound $a(H)/a(G) \leq i$ given in Proposition \ref{opensbgp}(ii) is sharp. 
\end{ex}

As in the previous section $w_p(n) = |n|_p^{-1}$ denotes the highest $p$-power dividing $n$. 
We begin with an upper bound result.
\begin{lem}\label{lem:upperbound-free-pro-p}
Let $q$ be a prime power and let $k \geq 0$. If $p$ divides $q-1$, then
\[r^*(F_r^p,\F_q,p^k) \leq w_p(q-1)^{p^k(r-1)+1}p^{\frac{(p^k-1)(r-1)}{p-1}}.\]
If $n$ is not a $p$-power or $n> 1$ and $p$ doesn't divide $q-1$, then $r^*(F_r^p,\F_q,n) = 0$.
\end{lem}
\begin{proof}
It is well-known that the degrees of absolutely irreducible representations of finite $p$-groups are $p$-powers. Since the center of the image of a non-trivial absolutely irreducible representation acts by a non-trivial character, such representations require the existence of $p$-th roots of units, i.e. $p \mid q-1$.

Write $n = p^k$ and assume $p\mid q-1$.
 Fix a Sylow $p$-subgroup $S_p$ of $\GL(n,q)$, which is isomorphic to
  \[C_{w_p(q-1)}\underbrace{ \wr C_p \wr \cdots \wr C_p}_{k \text{ times}},\] 
  and in general $|S_p| \leq w_p(q-1)^n p^{\frac{n-1}{p-1}}$: see \cite[Section 2]{Weir}. The claimed upper bound follows from Proposition \ref{upperbound-images}.
 \end{proof}

\begin{proof}[Proof of Theorem \ref{thm:abscissa-free-pro-p}]
 The upper bound $a(F^p_r) \leq \frac{r-1}{K'(p)}+1$ for the abscissa follows from Proposition \ref{prop:abscissa-Zp}
  in combination with $r^\ast(F^p_r,\F_{q},n) \leq q^{n(r-1)/K'(p)}w_p(q-1)$ (which can be deduced from Lemma \ref{lem:upperbound-free-pro-p}). 
  In fact, for real $s$ we have
  \begin{align*}
  	\log\zeta_{F_r^p}(s) &\leq \sum_{q,n} q^{n(r-1)/K'(p)}w_p(q-1) q^{-sn} \frac{q^{n}-1}{q-1}\\
	 &\ll \sum_{q} w_p(q-1)q^{-1} \sum_{n=1}^\infty q^{n\bigl((r-1)/K'(p) + (1-s)\bigr)}\\ 
	 &\ll  \sum_{q} w_p(q-1) q^{(r-1)/K'(p)+s} = \log\zeta_{\Z_p}\left(s-\frac{r-1}{K'(p)}\right)
  \end{align*}
  
Let $k\geq 1$. The cyclic group of order $p^k$ admits a $1$-dimensional faithful absolutely irreducible representation over $\F_q$ with $q = pp_{\min}(p^k)$ and the cyclic group of order $p$ admits a faithful permutation representation on $p$ elements. Corollary \ref{cor:lowerbound-general-perm} implies that
\[
	a(F_r^p) \geq \frac{\log(p^{k(p-1)+1})}{(p-1) \log(q)}(r-1) +1 = \frac{(k+\frac{1}{p-1})\log(p)}{\log(pp_{\min}(p^k))}(r-1) +1.
\]
Taking the supremum gives the lower bound.
 \end{proof}

\subsection{Free pro-\texorpdfstring{$2$}{2} groups}
\label{pro2}

We consider separately the case $p=2$. Let $F^2_r$ be the free pro-$2$ group on $r$ generators, with $r \geq 2$.
\begin{thm}\label{thm:abscissa-pro-2}
The abscissa of convergence for the free pro-$2$ group $F_r^2$ of rank $r \geq 2$ is $a(F_r^2) =\frac{5\log(2)}{2\log(3)}(r-1) +1$.
\end{thm}
We take an approach similar to that used for odd primes, however, the proof requires a number of modifications. We first discuss the upper bound  for $a(F^2_r)$. As usual we have $r^\ast(F^2_r,\F_{2^j},n) = 1$ for $n=1$ and $0$ otherwise.

\begin{lem}\label{lem:upperbound-pro-2}
Let $q$ be an odd prime power.
If $q \equiv 1 \bmod 4$, then
\[
	r^*(F_r^2,\F_q,2^k) \leq 2^{(2^k-1)(r-1)} w_2(q-1)^{2^k(r-1)+1}.
\]
If $q \equiv 3 \bmod 4$, then
\[
	r^*(F_r^2,\F_q,2^k) \leq 2^{(2^k+2^{k-1}-1)(r-1) +1} w_2(q+1)^{2^{k-1}(r-1)}.
\]
\end{lem}
\begin{proof}
Let $S$ be a Sylow $2$-subgroup of $\GL(2^k,q)$.
By Proposition \ref{upperbound-images}, we have
\[
	r^*(F_r^2,\F_q,2^k) \leq |S|^{r-1} |Z \cap S|
\]
where $Z \subseteq \GL_{2^k}(\F_q)$ denotes the group of scalar matrices.
The order of the Sylow $2$-subgroups of $\GL(2^k, q)$ can be found in \cite[Section 1]{CF}.
When $q \equiv 1 \bmod 4$ we have $|S| = 2^{2^k-1} w_2(q-1)^{2^k}$ and $|S \cap Z| = w_2(q-1)$.
When $q \equiv 3 \bmod 4$ we have $|S| = 2^{2^k+2^{k-1}-1} w_2(q+1)^{2^{k-1}}$ and $|S \cap Z| = 2$.
\end{proof}

This gives an upper bound for the zeta function:
\begin{align*}
	\log(\zeta_{F^2_r})(s) \leq \sum_{q,n} &\frac{\Lambda(q)}{\log(q)} 2^{(n-1)(r-1)}w_2(q-1)^{n(r-1)+1}\frac{q^n-1}{q-1}q^{-ns} \\
	+ 2\sum_{\substack{q \equiv 3 \bmod 4\\ n \text{ even}}} &\frac{\Lambda(q)}{\log(q)} 2^{(n+\frac{n}{2}-1)(r-1)}w_2(q+1)^{\frac{n}{2}(r-1)}\frac{q^n-1}{q-1}q^{-ns}.
\end{align*}

We show exactly as for $F^p_r$ with $p$ odd that the first sum converges when $s > (r-1)/K'(2)+1$ with
\[
	K'(2) = \inf_{k\geq 1} \frac{\log(pp_\text{min}(2^k))}{(k+ 1)\log(2)}.
\]
Recall that $pp_\text{min}(2^k)$ denotes the smallest prime power congruent $1$ modulo $2^k$. Similarly, let $pp_{\min}^-(2^k)$ denote the smallest prime power congruent $-1$ modulo $2^k$.
Define 
\[
K^- = \inf_{k \ge 1} \frac{2\log(pp_{\min}^-(2^k))}{(k+3)\log(2)}.
\]
 Then, for $q \equiv 3 \bmod4$ and $n$ even, $r^\ast(F^2_r,\F_{q},n) \leq q^{n(r-1)/K^-}$, so 
the second sum converges when $s > (r-1)/K^-+1$. As in Proposition \ref{Mersenne}, \cite[Theorem 1.6(i), Proposition 1.7]{Wolf} shows that $K'(2) \geq   K^- = \frac{2\log(3)}{5\log(2)} \approx 0.633985$.
Overall, we conclude that $a(F^2_r) \leq \frac{5\log(2)}{2\log(3)}(r-1)+1$.

We now turn to the lower bound. The proof for the upper bound suggests that the crucial prime power is $q = 3$ and it turns out that it is sufficient to consider this case. 

Here the $2$-dimensional representations will serve as our base case. 
A Sylow $2$-subgroup $S$ of $\GL(2,3)$ is a semidihedral group of order $16$. We observe that $S$ is absolutely irreducible, because it is non-abelian. Indeed, an abelian group cannot be absolutely irreducible in dimension $2$. Conversely, since $S$ is a non-abelian $2$-group, the representation on $\overline{\F_3}$ is semisimple and cannot decompose into $1$-dimensional pieces.
Using the action of the $2$-element group on two points, we deduce from Corollary \ref{cor:lowerbound-general-perm} that
\begin{align*}
a(F_r^2) \geq \frac{\log(16\cdot 2)}{2\log(3)} (r-1) +1 = \frac{5 \log(2)}{2 \log(3)}(r-1) +1.
\end{align*}

\subsection{Free prosoluble groups}\label{sec:free_prosolv}

In this section we discuss an upper bound for the abscissa of convergence of free prosoluble groups. Let $c_{sol} = \frac{2}{3}+\frac{5\log(2)}{2\log(3)} \approx 2.24399$.

\begin{thm}\label{thm:abscissa-sol}
Let $F^{sol}_r$ be the free prosoluble group on $r$ generators, $r \geq 2$.
Then $a(F_r^{sol}) = c_{sol} (r-1) + 1$.
\end{thm}

We collect some results needed in the proof.
By \cite[Theorem 3.1]{Wolf} and \cite[Corollary to Theorem A]{Segev} we have:
\begin{lem}
	\label{solubleorder}
	Let $q$ be a prime power.
	The order of a soluble irreducible subgroup of $\GL(n,q)$ is $\leq q^{c_{sol}n}$. For $q \geq 11$, the order of a soluble irreducible subgroup of $\GL(n,q)$ is $\leq q^{(1+\log_{q}(24)/3)n}$.
\end{lem}
We also use \cite[Theorem 14.1]{BNV}:

\begin{lem}\label{lem:max-soluble}
	The number of conjugacy classes of maximal irreducible soluble subgroups of $\GL(n,q)$ is at most $n^{4 \log^3(n)+4\log^2(n)+\log(n)+3}$.
\end{lem}

Fix a prime power $q$ and fix a maximal soluble subgroup $M$ in each conjugacy class. By Proposition \ref{upperbound-images}, the number of $\GL(n,q)$-conjugacy classes of homomorphisms to $M$ with absolutely irreducible image is at most $|M|^{r-1}(q-1)$. We conclude that 
\[
r^\ast(F^{sol}_r,\F_{q},n) \leq n^{4 \log^3(n)+4\log^2(n)+\log(n)+3} (q-1) q^{c_{sol}n(r-1)}
\]
 for all $q$, and 
\[
r^\ast(F^{sol}_r,\F_{q},n) \leq n^{4 \log^3(n)+4\log^2(n)+\log(n)+3} (q-1) q^{(1+\log_{q}(24)/3)n(r-1)}
\]
for $q \geq 11$.

For any $\varepsilon>0$, pick $Q$ such that $\log_Q(24)/3 < \varepsilon$. We split the series defining $\log(\zeta_{F^{sol}_r})(s)$ into terms with $q > Q$ and terms with $q \leq Q$. By Lemma \ref{generalbounds} and Lemma \ref{sameabscissa}, the first sum converges when $s > (1+\varepsilon)(r-1)+2$, while, for each $q \leq Q$, we have the series 
\[\sum_n \frac{n^{O(\log^3(n))}}{j} q^{(c_{sol}(r-1)+1-s)nj},\]
which converges when $s > c_{sol}(r-1)+1$. Since the second of these bounds is larger for $r \geq 2$, we conclude $a(F^{sol}) \leq c_{sol}(r-1)+1$.

The lower bound follows from Corollary \ref{cor:lowerbound-general-perm}. We note that $\GL(2,3)$ is an absolutely irreducible soluble group of order $48$ and that $S_4$ is a soluble group of order $24$. Corollary \ref{cor:lowerbound-general-perm} allows us to deduce
\[
	a(F_r^{sol})-1 \geq \frac{\log(48^3\cdot 24)}{2 \cdot 3 \cdot \log(3)}(r-1) = c_{sol}(r-1) 
\]

\subsection{Free pronilpotent groups}\label{sec:free_nilp}
The results for free pro-$p$ groups may be used to determine the abscissa for free pronilpotent groups. We note that the class of nilpotent groups is not closed under extensions, so that the results from Section \ref{sec:free-pro-c} cannot be applied directly. Let $F^{nil}_r = \prod_p F^p_r$ be the free pronilpotent group on $r$ generators, with $r \geq 2$.
\begin{thm}\label{thm:abscissa-pronilpotent}
Let $r \geq 2$. The abscissa of convergence of $\zeta_{F_r^{nil}}$ is 
\[
a(F_r^{nil}) = \begin{cases} 3 & \text{ if $r=2$}\\
			\frac{5 \log(2)}{2 \log(3)}(r-1) +1 & \text{if $r> 2$}.
			\end{cases}\]
\end{thm}
In other words, for $r \geq 3$ the free pro-$2$ factor of $F_r^{nil}$ is responsible for most of the representations of $F_r^{nil}$. In particular, the lower bound follows immediately from \ref{thm:abscissa-pro-2}. For $r=2$ the lower bound follows from $a(\widehat{\Z}^2) = 3$ obtained in Section \ref{sec:free-abelian-abscissa}.
 We will study the corresponding upper bound now.
 
 \begin{lem} Let $n \geq 1$ and
 let $q$ be a prime power.  Then
  \[
 r^\ast(F^{nil}_r,\F_{q},n)  \leq (q-1)^r q^{\frac{5\log(2)}{2\log(3)}(r-1)n}.\]
 \end{lem}
 \begin{proof}
 Let $n = p_1^{k_1} \cdots p_l^{k_l}$ be the prime decomposition. 
 Every absolutely irreducible representation of dimension $n$ over $\F_q$ is, by \cite{Fein}, of the form
 \[
 	\theta_1 \otimes \dots \otimes \theta_l \otimes \chi
 \]
 where $\theta_i$ is an absolutely irreducible representation of dimension $p_i^{k_i}$ of $F_r^{p_i}$ over $\F_q$ and 
 $\chi\colon \prod_{p \nmid n} F_r^{p} \to \F_q^\times$ is a homomorphism.
 Now Proposition \ref{upperbound-images} gives
 \begin{align*}
	r^\ast(F^{nil}_r,\F_{q},n) &\leq (q-1)^r \prod_{i=1}^l S(p_i,q)^{r-1} 
\end{align*}
where $S(p_i,q)$ is the order of a Sylow $p_i$-subgroup of $\GL(p_i^{k_i},q)$.

We note that $q$ and $n$ are coprime.
So we know by \cite[Theorem 1.6]{Wolf} that  $S(p_i,q) \leq q^{\frac{5\log(2)}{2\log(3)}p_i^{k_i}}$.
Therefore we have 
\[r^\ast(F^{nil}_r,\F_{q^j},n) \leq (q-1)^r q^{\frac{5\log(2)}{2\log(3)}(r-1) \sum_{i=1}^l p_i^{k_i}} \leq (q-1)^r q^{\frac{5\log(2)}{2\log(3)}(r-1)n}.\qedhere\]
 \end{proof}

\begin{proof}[Proof of the upper bounds in Theorem \ref{thm:abscissa-pronilpotent}]
We adapt the argument used for the free prosoluble group.
For any $\varepsilon>0$, pick $Q$ such that $\log_Q(24)/3 < \varepsilon$. We split the sum $\log(\zeta_{F^{nil}_r})(s)$ into terms with $q> Q$ and terms with $q \leq Q$. As for prosoluble groups we infer $r^\ast(F^{nil}_r,\F_{q},n) \leq n^{O(\log^3(n))}(q-1) q^{(1+\log_{q}(24)/3)n(r-1)}$ for $q \geq 11$ from Lemma~\ref{solubleorder} and Lemma~\ref{lem:max-soluble}. Thus we see by Lemma~\ref{generalbounds} and Lemma~\ref{sameabscissa} that the first sum converges when $s > (1+\varepsilon)(r-1)+2$, while, for each $q \leq Q$, we have the upper bound 
\[\sum_n (q-1)^{r} q^{(\frac{5\log(2)}{2\log(3)}(r-1)+1-s)n} \leq  Q^r\sum_n  q^{(\frac{5\log(2)}{2\log(3)}(r-1)+1-s)n}, \]
which converges when $s > \frac{5\log(2)}{2\log(3)} (r-1)+1$.

We note that $\frac{5\log(2)}{2\log(3)} > \frac{3}{2}$. Therefore,  when $r>2$, the second bound is larger, and we conclude $a(F^{nil}_r) \leq \frac{5\log(2)}{2\log(3)}(r-1)+1$. When $r=2$, the first bound is larger, and we conclude $a(F^{nil}_2) \leq 3$.
\end{proof}

\subsection{Pronilpotent groups of finite rank} \label{sec:finiterank}
Let $G$ be a profinite group. Recall that the rank of $G$ is the supremum of the minimal number of generators of all open subgroups; see \cite[3.11]{DDMS}.
\begin{prop}
Let $G$ be a pronilpotent group of finite rank $r$.
Then $a(G) \leq r+1$.
\end{prop}
The bound is sharp for the free abelian profinite group $\hat{\Z}^r$.
\begin{proof}
We consider the one-dimensional and higher dimensional representations separately. The one-dimensional representations factor through the abelianization, which itself is a factor of $\hat{\Z}^r$. Hence the sum of the contributions for all one-dimensional representations converges for $\Re(s) > r+1$; see \ref{sec:free-abelian-abscissa}. 

Let $n \geq 1$ and let $q$ be a prime power.
When $q \equiv 1 \bmod 4$, or when $n$ is odd, every absolutely irreducible representation of $G$ over $\F_{q}$ is monomial by \cite[Theorem 5.3]{DF2}. When $q \equiv 3 \bmod 4$ and $n$ is even, every absolutely irreducible representation of $G$ over $\F_{q^j}$ is induced from one of dimension $\leq 2$, by \cite[Theorem 5.1]{DF2}. We may count the absolutely irreducible primitive representations of dimension $2$ using the structure theory of \cite{DF}: by \cite[Proposition 4.2]{DF}, an absolutely irreducible primitive nilpotent subgroup $G$ of $\GL(2,q)$ is a direct product $G_2 \times C$ where $G_2$ is a primitive $2$-group and $C$ is a group of scalars of odd order.  If $q \equiv 3 \bmod 4$, then $4w_2(q+1)$ is the order of a Sylow $2$-subgroup of $\GL(2,q)$. 
Let $a_n(G)$ denote the number of subgroups of index $n$ in $G$. 
Taking into account the monomial representations and, if needed, those induced from $2$-dimensional representations, we deduce
\[
	r^*(G,\F_q,n) \leq a_n(G) (q-1)^r + \delta_{n,q} a_{n/2}(G) (4c_2(q+1))^r \Bigl(\frac{q-1}{2}\Bigr)^r 
\]
where $\delta_{n,q} \in \{0,1\}$ depending on whether $q \equiv 3 \bmod 4$ and $n$ is even. Since $G$ has finite rank, it has polynomial subgroup growth (see \cite[10.1]{LuSe}); hence, we get an upper bound of the form
\[
    r^*(G,\F_q,n)  \leq b_n q^{2r}
\]
where $b_n$ is a function which grows at most polynomially in $n$.

We use this to estimate the contribution to $\zeta_G$ from representations of dimension at least $2$ and obtain
\begin{align*}
\sum_{q,j} \sum_{n \geq 2} \frac{ r^*(G,\F_{q^j},n)}{j} q^{-jns} \frac{q^{jn}-1}{q^j-1}
&\leq \sum_{q,j} \frac{1}{j} \sum_{n \geq 2} b_n q^{2rj} q^{-jns} q^{j(n-1)}\\
&\leq \sum_{q,j} \frac{q^{(2r-1)j}}{j}\sum_{n \geq 2} b_n q^{j(1-s)n}
\end{align*}
The inner sum converges for $\Re(s) > 1$ and the result is bounded by $q^{2(1-s)j+\varepsilon}$ for any $\varepsilon > 0$ up to a constant. We deduce that the series converges for $\Re(s) > r+1+\varepsilon/2$. Since $\varepsilon$ was arbitrary, the result follows.
\end{proof}

\subsection{Free pro-\texorpdfstring{$\mathfrak{C}$}{C} groups, II}
\label{pro-C upper bound}

Let $\mathfrak{C}$ be a NE-formation of finite groups containing the cyclic groups of prime order. Let $F^{\mathfrak{C}}_r$ be the free pro-$\mathfrak{C}$ group on $r$ generators. If $\mathfrak{C}$ contains alternating groups of arbitrarily large degree, or classical groups with natural representation of arbitrarily large dimension, an application of Theorem \ref{thm:lowerbound-general} shows that $F^{\mathfrak{C}}_r$ does not have UBERG. Assume it does not. Let $c_0$ be maximal such that $Alt(c_0) \in \mathfrak{C}$, and assume that $c_0$ is large.

To make the calculation of the abscissa of convergence of $F^{\mathfrak{C}}_r$ easier, we start by showing that only certain types of absolutely irreducible representations need to be considered. Just as, for soluble groups, the abscissa is dominated by representations whose image is in a wreath product $GL(2,\F_3) \wr Sym(4) \wr \cdots \wr Sym(4)$, here the abscissa is dominated by representations whose image is in a wreath product $N_{\Gamma L_{\F_{p^j}}(\beta,\F_{p^k})}(E) \wr Sym(c_0) \wr \cdots \wr Sym(c_0)$, for $E$ a classical or alternating group with natural representation over $\F_{p^k}$ of dimension $\beta$. (Here, for an alternating group, by the natural representation we mean the fully deleted permutation module defined in \cite[Section 5.3, Alternating groups]{KL}.) Note that, when $E \in \mathfrak{C}$ is classical or alternating, $\Out(E)$ and $C_{\Gamma L_{\F_{p^j}}(\beta,\F_{p^k})}(E)$ are soluble, so $N_{\Gamma L_{\F_{p^j}}(\beta,\F_{p^k})}(E) \in \mathfrak{C}$ and hence $N_{\Gamma L_{\F_{p^j}}(\beta,\F_{p^k})}(E) \wr Sym(c_0) \wr \cdots \wr Sym(c_0)$ is too.

For $p$ a prime and $j,n \geq 1$ integers, consider the set of all classical or alternating groups $E \in \mathfrak{C}$ with natural representation over $\F_{p^k}$ of dimension $\beta$, such that $j|k$ and $n=\beta\frac{k}{j}c_0^l$, for some $l$.  We define $r^\ast_D(F^{\mathfrak{C}}_r,\F_{p^j},n)$ to be 
the sum over all such $E$ of $$|N_{\Gamma L_{\F_{p^j}}(\beta,\F_{p^k})}(E) \wr \underbrace{Sym(c_0) \wr \cdots \wr Sym(c_0)}_{\text{$l$ times}}|^{r-1}p^j.$$

Let $C_0$ be an absolute constant such that $(C_0!)^{1/(C_0-1)} \geq \max((2c_t)^4,2^{2c_4+4c_i+2})$, where $c_t,c_4,c_i$ are the (effectively computable) constants defined in \cite{JP}; we will refer to these constants without further comment in the rest of this section. We will assume for the rest of Section \ref{pro-C upper bound} that $c_0 > C_0$. (By Stirling's approximation, $(C_0!)^{1/(C_0-1)} \to \infty$ as $C_0 \to \infty$, so the required $C_0$ does exist.)

\begin{thm}
	\label{dominant-abscissa}
	The abscissa of convergence of $\zeta_{F^{\mathfrak{C}}_r}(s)$ is at most that of $\Sigma_D = \sum_{p,j,n}\frac{r^\ast_D(F^{\mathfrak{C}}_r,\F_{p^j},n)}{j}p^{-snj}|\mathbb{P}^{n-1}(\F_{p^j})|$.
\end{thm}

We will prove this in a series of smaller results.

\subsubsection{Permutation group results}

We recall following the well known result of \cite{PrSa}.

\begin{prop}
	\label{smallprim}
	Let $P$ be a primitive permutation group of degree $n$ and suppose that $P$ does not contain $Alt(n)$. Then the order of $P$ is at most $4^n$.
\end{prop}
\begin{prop}
	\label{transsize}
	If $G \leq Sym(n)$ is a transitive permutation group in $\mathfrak{C}$ then $\log(|G|) \leq \frac{n-1}{c_0-1}\log(c_0!)$, and the orders of the iterated wreath products $$S_k = \underbrace{Sym(c_0) \wr \cdots \wr Sym(c_0)}_{\text{$k$ times}} \leq Sym(c_0^k)$$ attain this bound. Thus, $|G| \leq (c_0!)^{(n-1)/(c_0-1)}$, and the supremum of $\log(|G|)/n$ over all $n$ and $G \in \mathfrak{C}$ is $\frac{\log(c_0!)}{c_0-1}$.
\end{prop}
\begin{proof}
	We have $G \leq P_1 \wr \cdots \wr P_l$ for some primitive $P_1, \ldots, P_l$ of degrees $s_1, \ldots, s_l$ respectively, with $P_i$ induced by the stabiliser in $G$ of a minimal block in the quotient action of $G$ on a set of $n/(s_1 \cdots s_{i-1})$ elements. If $s_1, \ldots, s_l \leq c_0$, the proof is a straightforward induction which we leave to the reader. Suppose instead that (without loss of generality) $s_1, \ldots, s_{i-1} \leq c_0$, and $s_i > c_0$. Write $G_i$ (respectively, $G_{i+1}$) for the image of $G$ in the quotient $P_i \wr \cdots \wr P_l$ (respectively, $P_{i+1} \wr \cdots \wr P_l$). By the inductive hypothesis, $\log(|G_{i+1}|) \leq \frac{s_{i+1} \cdots s_l-1}{c_0-1}\log(c_0!)$. If $P_i$ does not contain $Alt(s_i)$, $|P_i| \leq 4^{s_i}$ by Proposition \ref{smallprim}, and we calculate $\log(|G_i|) \leq s_i \cdots s_l \log(4) + \log(|G_{i+1}|) \leq s_i \cdots s_l \log(4) + \frac{s_{i+1} \cdots s_l-1}{c_0-1}\log(c_0!) \leq \frac{s_i \cdots s_l-1}{c_0-1}\log(c_0!)$. Indeed, because $s_i > c_0 > C_0 > 9$, this implies that $\frac{\log(c_0!)}{c_0-1} > \frac{c_0+1}{c_0}\log(4) \geq \frac{s_i}{s_i-1}\log(4)$, and hence $s_i \cdots s_l \log(4) \leq (s_i-1)s_{i+1} \cdots s_l \frac{\log(c_0!)}{c_0-1}$ and the result follows. If $P_i$ contains $Alt(s_i)$, consider $H = G_i \cap P_i^{s_{i+1} \cdots s_l}$: this is normal in $G_i$, so in $\mathfrak{C}$. Since $G$ permutes the factors transitively, if $H$ is non-trivial, the image of $H$ in each copy of $P_i$ is a non-trivial normal subgroup of $P_i$, so it contains $Alt(s_i)$, giving a contradiction. 
	Therefore $H$ is trivial, and $\log(|G_i|) \leq \frac{s_{i+1} \cdots s_l-1}{c_0-1}\log(c_0!) \leq \frac{s_i \cdots s_l-1}{c_0-1}\log(c_0!)$.
	
	Finally, $\log(|G|) \leq s_i \cdots s_l \log(|P_1 \wr \cdots P_{i-1}|) + \log(|G_i|)$, so by the inductive hypothesis, $\log(|G|) \leq \frac{\log(c_0!)}{c_0-1}((s_1 \cdots s_{i-1} - 1)s_i \cdots s_l + s_i \cdots s_l-1) = \frac{n-1}{c_0-1}\log(c_0!)$, as required.
	
	For the second claim, we have $|S_k| = |Sym(c_0)|^{c_0^{k-1} + c_0^{k-2} + \cdots + c_0^1}$, so $\log(|S_k|) = \frac{c_0^k-1}{c_0-1} \log(c_0!)$.
\end{proof}

In fact the proof shows more:

\begin{cor}
	\label{smallwithbigfactor}
	Suppose $G \leq P_1 \wr \cdots \wr P_l$ is in $\mathfrak{C}$, with $P_1, \ldots, P_l$ primitive permutation groups of degree $s_1, \ldots, s_l$. If $s_1, \ldots, s_{i-1} \leq c_0$ and $s_i > c_0$, $|G| \leq 4^{s_i \cdots s_l}(c_0!)^{\frac{n-s_i \cdots s_l + s_{i+1} \cdots s_l-1}{c_0-1}}$.
\end{cor}
\begin{proof}
	The proof of the proposition shows $\log(|G_i|) \leq s_i\log(4) + \frac{s_{i+1} \cdots s_l-1}{(c_0-1)}\log(c_0!)$. So $\log(|G|) \leq s_i \cdots s_l\log(|P_1 \wr \cdots \wr P_{i-1}|)+s_i \cdots s_l \log(4) + \frac{s_{i+1} \cdots s_l-1}{(c_0-1)} \log(c_0!) \leq \frac{n-s_i \cdots s_l + s_{i+1} \cdots s_l-1}{c_0-1} \log(c_0!) + s_i \cdots s_l \log(4)$.
\end{proof}

\begin{lem}
	\label{transcc}
	Let $s_1, \ldots, s_l$ be fixed. Suppose $G \in \mathfrak{C}$ is an $r$-generated transitive permutation group such that $G \leq P_1 \wr \cdots \wr P_l$ for some primitive $P_1, \ldots, P_l$ of degrees $s_1, \ldots, s_l$. Suppose $Alt(s_1), \ldots, Alt(s_{i-1}) \in \mathfrak{C}$ and $Alt(s_i) \notin \mathfrak{C}$. Then up to conjugacy in $Sym(n)$, for the constant $c_t$ defined in \cite[Theorem 3.1]{JP}, $G$ is contained in one of at most $c_t^{rs_i \cdots s_l}$ subgroups of $Sym(s_1) \wr \cdots \wr Sym(s_{i-1}) \wr P_i \wr \cdots \wr P_l$ which are in $\mathfrak{C}$.
\end{lem}
\begin{proof}
	By \cite[Theorem 3.1]{JP}, there are at most $c_t^{rs_i \cdots s_l}$ possibilities for the image $G'$ of $G$ in $P_i \wr \cdots \wr P_l$ up to conjugacy. The result follows immediately.
\end{proof}

\subsubsection{Linear group results}

Suppose $G \leq GL(n,p^j)$ is in $\mathfrak{C}$ and irreducible. We can write $G \leq P \wr T$, where $P \leq GL(b,p^j)$ is primitive, and $T \leq Sym(t)$, $t=n/b$, is transitive and in $\mathfrak{C}$ -- but note that in general $P$ need not be in $\mathfrak{C}$. We can also write $T \leq P_1 \wr \cdots \wr P_l$ as in the last section, with $P_1, \ldots, P_l$ primitive permutation groups of degrees $s_1, \ldots, s_l$ such that $s_1 \cdots s_l = t$.

In the case $|P|  > p^{c_4bj}$, we will need to fix some additional notation, following \cite[Proposition 5.7]{JP}. Suppose $\F_{p^{k'}} = Z(\End_{F^\ast(P)}(\F_{p^j}^b))$.  In this case, there exist $A \leq \Gamma L_{\F_{p^j}}(\alpha,\F_{p^{k'}}), B \leq \Gamma L_{\F_{p^j}}(\beta,\F_{p^{k'}})$ such that $P \leq A \odot B$ (so that $\alpha \beta k' = bj$ and $t=(nj)/(\alpha \beta k')$). These have the property that $\F_{p^{k'}}^\times \leq A$, $\F_{p^{k'}}^\times \leq B$, and $\beta \geq \alpha$, so that $|A| \leq p^{(\alpha^2+1)k'} \leq p^{bj+k'}$. Also, $E(B)$ is either a classical group over a subfield $\F_{p^k}$ of $\F_{p^{k'}}$ with natural representation in dimension $\beta$ with scalars extended to $\F_{p^{k'}}$, or an alternating group with natural representation over $\F_{p^{k'}}$ in dimension $\beta$.

For each classical group $E \in \mathfrak{C}$ over $\F_{p^k}$ with natural representation in dimension $\beta$, define $c_E$ by $|N_{\Gamma L_{\F_p}(\beta,\F_{p^k})}(E)| = p^{c_E \beta k}$ (so that $|N_{\Gamma L_{\F_{p^j}}(\beta,\F_{p^k})}(E)| = \frac{1}{j}p^{c_E \beta k}$). 
For each alternating group $E \in \mathfrak{C}$ with natural representation over $\F_{p^k}$ in dimension $\beta$, define $c_{E,p^k}$ by $|N_{\Gamma L_{\F_p}(\beta,\F_{p^k})}(E)| = p^{c_{E,p^k} \beta k}$. We have $|N_{\Gamma L_{\F_p}(\beta,\F_{p^k})}(Alt(d))| = |\Out(Alt(d))|k(p^k-1)$ (with $\Out(Alt(d)) = Sym(d)$ for $d \neq 6$), and $\beta = d-1$ or $d-2$,  so an easy calculation shows that $c_{Alt(d),p^k} < c_{Alt(d),2}$ for any $p^k>2$ and $d \geq 5$. When the choice of field is clear, we may suppress the dependence on $p^k$ in the notation.

We want to count representations of $F^{\mathfrak{C}}_r$ with image $G$. We consider several classes of possibilities for $G$. The strategy for each class is to show that $G$ is contained in some larger subgroup of $GL(n,p^j)$ in $\mathfrak{C}$ which has an easy-to-describe form, such that we can bound both the order of these larger subgroups and the number of such subgroups up to conjugacy.
\\

1. $|P| \leq p^{c_4bj}$.
\\

\begin{prop}
	The number $r^\ast_1(F^{\mathfrak{C}}_r,\F_{p^j},n)$ of absolutely irreducible representations of $F^{\mathfrak{C}}_r$ with image in class 1 is at most 
	\begin{equation}\label{eq:f1}
	  p^{c_4(r-1)nj+j+c_irnj}(c_0!)^{\frac{(r-1)(n/2-1)}{c_0-1}}.
	\end{equation}
\end{prop}
\begin{proof}
	Suppose $G$ is in class 1. By Proposition \ref{transsize}, $|G| \leq p^{c_4nj}(c_0!)^{\frac{t-1}{c_0-1}}$, where $t \leq n/2$. By \cite[Proposition 6.1]{JP}, up to conjugacy in $GL(n,p^j)$, there are at most $p^{c_irnj}$ possibilities for the image $G$ of $F^{\mathfrak{C}}_r$. The result follows by Proposition \ref{upperbound-images}.
\end{proof}
\bigskip

2. $G$ is not in class 1, and $E(B)$ is not in $\mathfrak{C}$.
\\

\begin{prop}
	The number $r^\ast_2(F^{\mathfrak{C}}_r,\F_{p^j},n)$ of absolutely irreducible representations of $F^{\mathfrak{C}}_r$ with image in class 2 is at most $p^{2(r-1)nj+j+c_irnj}(c_0!)^{\frac{(r-1)(n/2-1)}{c_0-1}}$.
\end{prop}
\begin{proof}
	Suppose $G$ is in class 2. Consider $H = G \cap E(B)^t$, which is normal in $G$, so in $\mathfrak{C}$. If $H \nleq Z(E(B))^t$, since the factors are permuted transitively, the image of $H$ in each copy of $E(B)$ is a non-trivial normal subgroup of $E(B)$, so it contains $E(B)$, giving a contradiction. Therefore $H \leq Z(E(B))^t$. Also, $|N_{\Gamma L_{\F_{p^j}}(\beta,\F_{p^{k'}})}(E(B))/E(B)|$ is at most  
	\[
	|C_{\Gamma L_{\F_{p^j}}(\beta,\F_{p^{k'}})}(E(B))| |\Out(E(B)/Z(E(B)))| \leq \frac{k'}{j} p^{2bj} \leq \frac{1}{j} p^{3bj} 
	\]
by \cite[Lemma 2.6, Proposition 5.7]{JP}. So, by Proposition \ref{transsize}, $|G| \leq \frac{1}{j^t} p^{3nj} (c_0!)^{\frac{t-1}{c_0-1}}$, with $t \leq n/2$.
	
	As for class 1, there are at most $p^{c_irnj}$ possibilities for $G$. The result follows by Proposition \ref{upperbound-images}.
\end{proof}
\bigskip

3. $G$ is not in classes 1 or 2.
\\

In this case, we have $$|B| \leq |N_{\Gamma L_{\F_{p^j}}(\beta,\F_{p^{k'}})}(E(B))| \leq \frac{k'}{kj} p^{c_{E(B)} \beta k + k'} \leq \frac{1}{j} p^{c_{E(B)} \beta k + 2k'}.$$ Moreover, $B \in \mathfrak{C}$, because $C_{GL(\beta,\F_{p^{k'}})}(E(B))$ is cyclic and $\Out(E(B))$ is soluble.

We divide class 3 into classes 3.$E$ of groups $G$ in class 3 such that $E(B) = E$, over all classical and alternating groups $E$ in $\mathfrak{C}$. We subdivide these classes further.
\\

3.$E$.1. $\alpha > 1$.
\\

\begin{prop}
	The number $r^\ast_{E.1}(F^{\mathfrak{C}}_r,\F_{p^j},n)$ of absolutely irreducible representations of $F^{\mathfrak{C}}_r$ with image in class 3.$E$.1 is at most $$\frac{nj}{\beta k} c_t^{rnj/(2\beta k)} (\frac{1}{j})^{(r-1)(nj)/(\beta k)} p^{(c_E/2+2)(r-1)nj+j} (c_0!)^{\frac{(r-1)((nj)/(2\beta k)-1)}{c_0-1}}.$$
\end{prop}
\begin{proof}
	Since $\beta k' \leq bj/2$, we have $|B| \leq \frac{1}{j} p^{c_{E(B)} bj/2 + bj}$.  Therefore $|P| \leq |A \odot B| = |A| |B| /p^{k'} \leq \frac{1}{j} p^{c_{E(B)} bj/2 + 2bj}$, and hence, by Proposition \ref{transsize}, $$|G| \leq (\frac{1}{j})^{(nj)/(\beta k')} p^{(c_{E(B)}/2+2)nj} (c_0!)^{\frac{(nj)/(2\beta k')-1}{c_0-1}}.$$
	
	In this class, we have $G \leq (\Gamma L_{\F_{p^j}}(\alpha,\F_{p^{k'}}) \odot N_{\Gamma L_{\F_{p^j}}(\beta,\F_{p^{k'}})}(E(B))) \wr T$. There are at most $\frac{nj}{\beta k}$ choices for $\alpha$ and then, by \cite[Theorem 3.1]{JP}, at most $c_t^{rnj/(2\beta k)}$ choices for $T$ up to conjugacy, giving $\frac{nj}{\beta k}c_t^{rnj/(2\beta k)}$ possibilities altogether. The result follows by Proposition \ref{upperbound-images}.
\end{proof}
\bigskip

3.$E$.2. $G$ is not in class 3.$E$.1, and $k'>k$.
\\

\begin{prop}
	The number $r^\ast_{E.2}(F^{\mathfrak{C}}_r,\F_{p^j},n)$ of absolutely irreducible representations of $F^{\mathfrak{C}}_r$ with image in class 3.$E$.2 is at most $$c_t^{rnj/(\beta k)} (\frac{1}{j})^{(r-1)(nj)/(\beta k)} p^{(c_E/2 + 3)(r-1) nj+j} (c_0!)^{\frac{(r-1)((nj)/(\beta k)-1)}{c_0-1}}.$$
\end{prop}
\begin{proof}
	Here we may choose $B$ to be $P$, so $$|P|=|B| \leq \frac{1}{j} p^{c_{E(B)} \beta k + 2k'} \leq \frac{1}{j} p^{(c_{E(B)}/2 + 2)\beta k'} = \frac{1}{j} p^{(c_{E(B)}/2+2)bj},$$ and hence $|G| \leq (\frac{1}{j})^{(nj)/(\beta k')} p^{(c_{E(B)}/2 + 2) nj} (c_0!)^{\frac{(nj)/(\beta k')-1}{c_0-1}}$ by Proposition \ref{transsize}.
	
	In this class, we have $G \leq N_{\Gamma L_{\F_{p^j}}(\beta,\F_{p^{k'}})}(E(B))) \wr T$. There are at most $c_t^{rnj/(\beta k)}$ choices for $T$ up to conjugacy, by \cite[Theorem 3.1]{JP}, and hence at most $c_t^{rnj/(\beta k)}$ possibilities for $N_{\Gamma L_{\F_{p^j}}(\beta,\F_{p^{k'}})}(E(B))) \wr T$. The result follows by Proposition \ref{upperbound-images}.
\end{proof}
\bigskip

3.$E$.3. $G$ is not in classes 3.$E$.1 or 3.$E$.2, and some $s_i$ is greater than $c_0$.
\\

\begin{prop}
	The number $r^\ast_{E.3}(F^{\mathfrak{C}}_r,\F_{p^j},n)$ of absolutely irreducible representations of $F^{\mathfrak{C}}_r$ with image in class 3.$E$.3 is at most 
	\begin{multline*}
	 2^{\log_2((nj)/(\beta k))^2} c_t^{rs_i \cdots s_l} (\frac{1}{j})^{(r-1)(nj)/(\beta k)} p^{c_E (r-1) nj + j} 4^{(r-1)s_i \cdots s_l} \cdot \\ \cdot (c_0!)^{\frac{(r-1)((nj)/(\beta k)-s_i \cdots s_l + s_{i+1} \cdots s_l-1)}{c_0-1}}.
	\end{multline*}
\end{prop}
\begin{proof}
	Note that $l \leq \log_2((nj)/(\beta k))$, and there are at most $(nj)/(\beta k)$ choices for each $s_i$, so there are at most $((nj)/(\beta k))^{\log_2((nj)/(\beta k))} = 2^{\log_2((nj)/(\beta k))^2}$ possibilities for $s_1,s_2,\ldots,s_l$. Fix one possibility.
	
	Without loss of generality, we assume that $s_1, \ldots, s_{i-1} \leq c_0$. Because $k=k'$, we have 
	$|P| \leq \frac{1}{j} p^{c_{E(B)} bj}$. By Corollary \ref{smallwithbigfactor}, we have $$|G| \leq (\frac{1}{j})^{(nj)/(\beta k)} p^{c_{E(B)} n j} 4^{s_i \cdots s_l}(c_0!)^{\frac{(nj)/(\beta k)-s_i \cdots s_l + s_{i+1} \cdots s_l-1}{c_0-1}}.$$
	
	By Lemma \ref{transcc}, for fixed $s_1, \ldots, s_l$, $G$ is contained in one of at most $c_t^{rs_i \cdots s_l}$ possibilities in $\mathfrak{C}$ in this class, up to conjugacy. The result follows by Proposition \ref{upperbound-images}.
\end{proof}
\bigskip

3.$E$.4. $G$ is not in classes 3.$E$.1 or 3.$E$.2 or 3.$E$.3.
\\

\begin{prop}
	\label{E.4}
	The number $r^\ast_{E.4}(F^{\mathfrak{C}}_r,\F_{p^j},n)$  of absolutely irreducible representations of $F^{\mathfrak{C}}_r$ with image in class 3.$E$.4 is at most $$2^{\log_2((nj)/(\beta k))^2} (\frac{1}{j})^{(r-1)(nj)/(\beta k)} p^{c_E(r-1) nj+j} (c_0!)^{\frac{(r-1)((nj)/(\beta k)-1)}{c_0-1}}.$$
\end{prop}
\begin{proof}
	As before, there are at most $2^{\log_2((nj)/(\beta k))^2}$ possibilities for $s_1, \ldots, s_l$. Fix one. Then $G \leq N_{\Gamma L_{\F_{p^j}}(\beta,\F_{p^{k'}})}(E(B))) \wr Sym(s_1) \wr \cdots \wr Sym(s_l)$, so $$|G| \leq (\frac{1}{j})^{(nj)/(\beta k)} p^{c_{E(B)} nj} (c_0!)^{\frac{(nj)/(\beta k)-1}{c_0-1}}$$ by Proposition \ref{transsize}. The result follows by Proposition \ref{upperbound-images}.
\end{proof}

\begin{proof}[Proof of Theorem \ref{dominant-abscissa}]
	Recall, for $G$ in class 3.$E$, that $c_{E(B)} > c_4 > 4$, and that $Alt(s_i) \in \mathfrak{C}$ for $s_i < 5$. Since $(c_0!)^{1/(c_0-1)} \geq (2c_t)^4$, a calculation shows that the bounds for $r^\ast_{E.1}(F^{\mathfrak{C}}_r,\F_{p^j},n), r^\ast_{E.2}(F^{\mathfrak{C}}_r,\F_{p^j},n)$ and $r^\ast_{E.3}(F^{\mathfrak{C}}_r,\F_{p^j},n)$ are at most 
	\begin{equation}\label{eq:f2}
	 2^{\log_2((nj)/(\beta k))^2} (\frac{1}{j})^{(r-1)(nj)/(\beta k)} p^{c_E(r-1) nj+j} (c_0!)^{\frac{(r-1)((nj)/(\beta k)-1)}{c_0-1}}.
	\end{equation}
  
    Also, it is clear that the upper bound for $r^\ast_2(F^{\mathfrak{C}}_r,\F_{p^j},n)$ is less than the upper bound for $r^\ast_1(F^{\mathfrak{C}}_r,\F_{p^j},n)$.

	We have shown that $r^\ast(F^{\mathfrak{C}}_r,\F_{p^j},n)$ is at most
	\begin{align} \label{line1}
	&  2p^{c_4(r-1)nj+j+c_irnj}(c_0!)^{\frac{(r-1)(n/2-1)}{c_0-1}} + \\ &+ 4\sum_{E \in \mathfrak{C}} 2^{\log_2((nj)/(\beta k))^2} (\frac{1}{j})^{(r-1)(nj)/(\beta k)} p^{c_E(r-1) nj+j} (c_0!)^{\frac{(r-1)((nj)/(\beta k)-1)}{c_0-1}}, \label{line2}
	\end{align}
	 so $\log(\zeta_G) \leq 2\Sigma_1 + 4\Sigma_2$ where
	\begin{align*}
		\Sigma_1 &= \sum_{p \in \mathcal{P}} \sum_{j,n=1}^\infty \frac{f_1(p,n,j) 
		}{j} p^{-snj} |\mathbb{P}^{n-1}(\F_{p^j})|, \\
		\Sigma_2 &= \sum_{E \in \mathfrak{C}} \sum_{p \in \mathcal{P}} \sum_{j,n=1}^\infty \frac{f_2(p,n,j)
		}{j} p^{-snj} |\mathbb{P}^{n-1}(\F_{p^j})|,
	\end{align*}
	$f_1(p,n,j)$ and $f_2(p,n,j)$ are the functions appearing in \eqref{line1} and \eqref{line2}. Therefore, $a(G)$ is at most the maximum of the abscissae of $\Sigma_1$ and $\Sigma_2$. We have $Sym(c_0) \in \mathfrak{C}$, with (absolutely irreducible) natural representation over $\F_2$ of dimension $c_0-\delta(c_0)$, where $\delta(c_0)=1$ or $2$ for odd or even $c_0$, respectively. By Theorem \ref{thm:lowerbound-general}, we have $a(F_r^{\mathfrak{C}}) \geq \frac{\log(c_0!)}{(c_0-\delta(c_0))\log(2)}(r-1) +1$. On the other hand, the abscissa of $\Sigma_1$ is at most $c_4(r-1)+c_ir+\frac{\log(c_0!)(r-1)}{2\log(2)(c_0-1)}+2$ by Lemma \ref{generalbounds}. So when $(c_0!)^{1/(c_0-1)} > 2^{2c_4+4c_i+2}$, a calculation shows that we have $a(G)$ larger than the abscissa of $\Sigma_1$, and hence at most that of $\Sigma_2$.
			
	For each $E \in \mathfrak{C}$, and each $n = \beta \frac{k}{j} c_0^l$, some $l$, we have $$|N_{\Gamma L_{\F_{p^j}}(\beta,\F_{p^k})}(E) \wr Sym(c_0) \wr \cdots \wr Sym(c_0)| = (\frac{1}{j})^{(nj)/(\beta k)} p^{c_Enj}(c_0!)^{\frac{(nj)/(\beta k)-1}{c_0-1}},$$ by Proposition \ref{transsize}. So the sum $$\sum_{(nj)/(\beta k)=1}^{c_0^l}  2^{\log_2((nj)/(\beta k))^2} (\frac{1}{j})^{(r-1)(nj)/(\beta k)} p^{c_E(r-1) nj+j} (c_0!)^{\frac{(r-1)((nj)/(\beta k)-1)}{c_0-1}}$$ is at most $c_0^l 2^{\log_2(c_0^l)^2} |N_{\Gamma L_{\F_{p^j}}(\beta,\F_{p^k})}(E) \wr Sym(c_0) \wr \cdots \wr Sym(c_0)|^{r-1} p^j$. It follows, using the usual techniques, that the abscissa of $\Sigma_2$ is at most that of $\Sigma_D$, by summing both over $n$ for a fixed $E$ and $j$, and then by summing over $E$ and $j$.
\end{proof}

In fact, we can do slightly better: for $p$ a prime and $n \geq 1$ an integer, consider the set $S_{p,n}$ of all classical or alternating groups $E \in \mathfrak{C}$ with natural representation over $\F_{p^k}$ of dimension $\beta$, such that $n=\beta k c_0^l$, for some $l$. 
We define $r^\ast_{D'}(F^{\mathfrak{C}}_r,\F_p,n)$ to be the sum over $S_{p,n}$ of $$|N_{\Gamma L(\beta,\F_{p^k})}(E) \wr \underbrace{Sym(c_0) \wr \cdots \wr Sym(c_0)}_{\text{$l$ times}}|^{r-1}p.$$

\begin{cor}
	\label{dominant-abscissa2}
	The abscissa of convergence of $\zeta_{F^{\mathfrak{C}}_r}(s)$ is at most that of $\Sigma_{D'} = \sum_{p,n}r^\ast_{D'}(F^{\mathfrak{C}}_r,\F_p,n)p^{-ns}|\mathbb{P}^{n-1}(\F_p)|$.
\end{cor}
\begin{proof}
	It is enough to show that $\Sigma_{D'}$ has the same abscissa as $\Sigma_D$. Clearly $\Sigma_{D'} \leq \Sigma_D$, so one direction is trivial. We prove the converse.
	
	For a fixed classical or alternating group $E$ with natural representation over $\F_{p^k}$ of dimension $\beta$, and fixed $l$ (so that $nj =\beta k c_0^l$), the contribution to $\Sigma_D$ is 
	\begin{multline*}
	\sum_{j:j|k} \frac{|N_{\Gamma L_{\F_{p^j}}(\beta,\F_{p^k})}(E) \wr Sym(c_0) \wr \cdots \wr Sym(c_0)|^{r-1}p^j}{j} p^{-snj} |\mathbb{P}^{n-1}(\F_{p^j})| \leq \\ \leq 2\sum_{j:j|k} |N_{\Gamma L_{\F_{p^j}}(\beta,\F_{p^k})}(E) \wr Sym(c_0) \wr \cdots \wr Sym(c_0)|^{r-1} p^{-s \beta k c_0^l} (p^{\beta k c_0^l -1}) \leq \\ \leq 2k |N_{\Gamma L(\beta,\F_{p^k})}(E) \wr Sym(c_0) \wr \cdots \wr Sym(c_0)|^{r-1} p^{-s \beta k c_0^l} (p^{\beta k c_0^l -1}),
	\end{multline*} 
	which is at most $2k$ times the contribution of this fixed $E$ and $l$ to $\Sigma_{D'}$. Summing over $E$ and $l$, Lemma \ref{sameabscissa} now gives the result.
\end{proof}

Let $c_{\mathfrak{C}, \text{space}}$ be the supremum of $c_E+\frac{\log_p(c_0!)}{\beta k (c_0-1)}$ over all classical and alternating groups $E \in \mathfrak{C}$. For a constant $K$, let $c_{\mathfrak{C}, \text{time},K}$ be the supremum of $c_E$ over all classical and alternating groups $E$ such that $p^k \geq K$, and let $c_{\mathfrak{C}, \text{time}} = \inf_K(c_{\mathfrak{C}, \text{time},K})$; clearly $c_{\mathfrak{C}, \text{time}} \leq c_{\mathfrak{C}, \text{time},K} \leq c_{\mathfrak{C}, \text{space}}$ for any $K$.

\begin{thm}
	\label{spacetime-bounds}
	$c_{\mathfrak{C}, \text{space}}(r-1)+1 \leq a(F^{\mathfrak{C}}_r) \leq \max(c_{\mathfrak{C}, \text{space}}(r-1)+1, c_{\mathfrak{C}, \text{time}}(r-1)+2)$.
\end{thm}
\begin{proof}
	The lower bound follows from Corollary \ref{cor:lowerbound-general-perm} by taking, for all $\epsilon>0$, $S=N_{\Gamma L_{\F_p}(n,\F_{p^k})}(E)$ for some $E \in \mathfrak{C}$ with $c_E+\frac{\log_p(c_0!)}{\beta k (c_0-1)} > c_{\mathfrak{C}, \text{space}}-\epsilon$, and $T=Sym(c_0)$.
	
	The number of groups $E$ contributing to $r^\ast_D(F^{\mathfrak{C}}_r,\F_{p^j},n)$ is at most $7 n^2$: there are at most $n$ choices for each of $\beta$ and $k$, and then at most 7 choices of classical or alternating group. So $r^\ast_D(F^{\mathfrak{C}}_r,\F_{p^j},n) \leq 7n^2 p^{c_{\mathfrak{C}, \text{space}} (r-1)nj+j}$.
	
	To prove the upper bound, we may split our sum 
	into the sum $$\sum_{p,j,n: p^j \geq K}\frac{r^\ast_D(F^{\mathfrak{C}}_r,\F_{p^j},n)}{j}p^{-njs}\frac{p^{nj}-1}{p^j-1}$$ and finitely many sums of the form $$\sum_n\frac{r^\ast_D(F^{\mathfrak{C}}_r,\F_{p^j},n)}{j}p^{-njs}\frac{p^{nj}-1}{p^j-1}$$ with $p^j$ fixed. Each of the sums with $p^j$ fixed converges for $s > c_{\mathfrak{C}, \text{space}} (r-1) + 1$. On the other hand, $\frac{\log_p(c_0!)}{\beta k (c_0-1)} \to 0$ as $p^k \to \infty$. So for all $\epsilon > 0$, we can pick $K$ large enough that for $p^k \geq K$, we have $r^\ast_D(F^{\mathfrak{C}}_r,\F_{p^j},n) \leq 7n^2 p^{(c_{\mathfrak{C}, \text{time}}+\epsilon)(r-1)nj+j}$. Since $\epsilon$ is arbitrary, we get the result by Lemma \ref{generalbounds} and Lemma \ref{sameabscissa}.
\end{proof}

As examples, we will now calculate the abscissa in more detail for two specific NE-formations. Let $\mathfrak{C}_{Alt(c_0)}$ be the NE-formation generated by the cyclic groups of prime order and the alternating groups of degree $\leq c_0$; let $\mathfrak{C}_{\Sigma(c_0)}$ be the NE-formation generated by the class of groups defined in \cite{BCP}: that is, all the simple groups in $\mathfrak{C}_{Alt(c_0)}$, all exceptional groups of Lie type, all non-abelian simple groups of order at most $c_0$, and all classical simple groups whose natural representation has dimension at most $c_0$. None of the calculations below would be affected if we followed \cite[Window: Permutation groups, Section 2]{LuSe} in also including all the sporadic simple groups (this is not done in \cite{BCP} only because the classification of finite simple groups was incomplete when the paper was written), so our results also hold in this case.

\subsubsection{$\mathfrak{C}_{Alt(c_0)}$}

Let $\delta(c_0)=1$ or $2$ for odd or even $c_0$, respectively.

\begin{thm}
	Suppose that $c_0 > C_0$. For any $r \geq 2$, $$a(F^{\mathfrak{C}_{Alt(c_0)}}_r) = \frac{c_0\log_2(c_0!)}{(c_0-\delta(c_0))(c_0-1)}(r-1)+1.$$
\end{thm}
\begin{proof}
	Recall that, for $E$ alternating, of degree $\geq 5$, that $c_{E,p^k}$, over all fields $\F_{p^k}$, is maximised when $p^k=2$. In this case, we have $|N_{\Gamma L_{\F_p}(\beta,\F_{p^k})}(E)| = c_0!$, and $\beta = c_0-\delta(c_0)$. So $c_0! = 2^{c_{Alt(c_0)}(c_0-\delta(c_0))}$ and $c_{Alt(c_0)} = \frac{\log_2(c_0!)}{c_0-\delta(c_0)}$.
	
	In fact, $\frac{\log_p(c_0!)}{\beta k(c_0-1)}$ is also maximised in this case, giving $\frac{\log_2(c_0!)}{\beta(c_0-1)}$. So $c_{\mathfrak{C}_{Alt(c_0)}, \text{space}} = \sup_{a \leq c_0} (\frac{\log_2(a!)}{a-\delta(a)} + \frac{\log_2(c_0!)}{(a-\delta(a))(c_0-1)}) = \frac{\log_2(c_0!)}{c_0-\delta(c_0)} + \frac{\log_2(c_0!)}{(c_0-\delta(c_0))(c_0-1)} = \frac{c_0\log_2(c_0!)}{(c_0-\delta(c_0))(c_0-1)}$.
	
	On the other hand, the same proof shows that for $p^k \geq K$ large, $c_{E,p^k} \leq \frac{\log_K(c_0!)+1}{c_0-2}$. So $c_{\mathfrak{C}_{Alt(c_0)}, \text{time}} \leq \frac{1}{c_0-2}$, and hence, for any $r \geq 2$, $$a(F^{\mathfrak{C}_{Alt(c_0)}}_r) \leq \frac{c_0\log_2(c_0!)}{(c_0-\delta(c_0))(c_0-1)}(r-1)+1.$$ This is also the lower bound for $a(F^{\mathfrak{C}_{Alt(c_0)}}_r)$ given by Theorem \ref{spacetime-bounds}.
\end{proof}

Stirling's approximation shows that $\frac{c_0\log_2(c_0!)}{(c_0-\delta(c_0))(c_0-1)} \sim \log_2(c_0)$ as $c_0 \to \infty$.

\begin{rmk}
	In fact, this approach can be used to show that we could modify the definition of $\Sigma_{D'}$ to consider, in the contributions of alternating groups to the sum, only their natural representations over $\F_2$, and the statement of Corollary \ref{dominant-abscissa2} would still hold.
\end{rmk}

\subsubsection{$\mathfrak{C}_{\Sigma(c_0)}$}

For a prime power $p^k$, define $\rho(p^k) = \frac{(c_0-1)\log(k\prod_{i=1}^{c_0}(1-p^{-ik}))+\log(c_0!)}{c_0(c_0-1)\log(p^k)}$.

\begin{thm}
	Suppose that $c_0 > C_0$. For any $r \geq 2$, $a(F^{\mathfrak{C}_{\Sigma(c_0)}}_r) = (c_0+\rho(2))(r-1)+1$. Moreover, $\frac{\log(c_0)}{(c_0-1)\log(2)} - \frac{3}{c_0 \log(2)} < \rho(2) < \frac{\log(c_0)}{(c_0-1)\log(2)}$.
\end{thm}
\begin{proof}
	By hypothesis, for any prime power $p^k$ and any classical or alternating $E \in \mathfrak{C}_{\Sigma(c_0)}$ defined over $\F_{p^k}$, $E$ has natural representation of dimension $\beta$, some $\beta \leq c_0$. Since $GL(c_0,\F_{p^k}) \in \mathfrak{C}_{\Sigma(c_0)}$, it is straightforward to show that, for any prime power $p^k$, $|N_{\Gamma L_{\F_p}(\beta,\F_{p^k})}(E)| \leq p^{c_{SL(c_0,\F_{p^k})}\beta k}$ (see \cite[Table 2.1.C]{KL}). We calculate $|\Gamma L_{\F_p}(c_0,\F_{p^k})| = kp^{c_0^2k} \prod_{i=1}^{c_0}(1-p^{-ik}) = p^{c_{SL(c_0,\F_{p^k})}c_0k}$. As $p^k \to \infty$, $c_{p^k} \to c_0$, so $c_{\mathfrak{C}_{\Sigma(c_0)},\text{time}} = c_0$. Similarly, we see that $\frac{\log_p(|N_{\Gamma L_{\F_p}(\beta,\F_{p^k})}(E)|)}{\beta k} + \frac{\log_p(c_0!)}{\beta k(c_0-1)}$  is maximised, for each $p^k$, by $\beta=c_0$ and $E = SL(c_0,\F_{p^k})$, so $c_{\mathfrak{C}_{\Sigma(c_0)},\text{space}}$ is given by $\sup_{p^k}(c_{SL(c_0,\F_{p^k})} + \frac{\log_p(c_0!)}{kc_0(c_0-1)}) = c_0 + \sup_{p^k}\rho(p^k)$.
	
	We split the sum $\sum_{p,j,n}\frac{r^\ast_D(F^{\mathfrak{C}}_r,\F_{p^j},n)}{j}p^{-njs}\frac{p^{nj}-1}{p^j-1}$ as follows. For any $\epsilon > 0$, set $K$ large enough that $\rho(p^k) < \epsilon$ for $p^j \geq K$. For each $p^j < K$, $\Sigma_1(p^j)$ is the sum over $n$ of terms with $p^j$ fixed. $\Sigma_2$ is the sum over $p,j,n$ of terms with $p^j \geq K$ and $n \geq c_0$. For each $n < c_0$, $\Sigma_3(n)$ is the sum over $p,j$ such that $p^j \geq K$, with $n$ fixed. This partitions the terms of $\sum_{p,j,n}\frac{r^\ast_D(F^{\mathfrak{C}}_r,\F_{p^j},n)}{j}p^{-njs}\frac{p^{nj}-1}{p^j-1}$ into finitely many parts.
	
	Each $\Sigma_1(p^j)$ converges for $s > c_{\mathfrak{C}_{\Sigma(c_0)},\text{space}}(r-1)+1 = (c_0+\sup_{p^k}\rho(p^k))(r-1)+1$.
	
	Next, $\Sigma_2 \leq \sum_{p,j,n: p^j \geq K, n \geq c_0}  7n^2 2^{\log_2(nj)^2} p^{((c_0+\epsilon)(r-1)+1-s)nj}$. This has the same abscissa of convergence as $\sum_{p,j,n: p^j \geq K, n \geq c_0} p^{-tnj}$, by Lemma \ref{sameabscissa}, for $s=(c_0+\epsilon)(r-1)+1+t$. Summing over $n$ first, we get $\sum_{p,j:p^j \geq K} \frac{p^{-tc_0j}}{1-p^{-tj}} \ll \sum_{p,j:p^j \geq K} p^{-tc_0j}$, where the implicit constant depences only on $t$; the latter sum converges when $t > 1/c_0$.
	
	Finally, for each $n < c_0$, $|\Gamma L_{\F_p}(n,p^j)| \leq p^{(n+\epsilon)nj}$. As before, we get that $\Sigma_3(n)$ has the same abscissa of convergence as $\sum_{p,j:p^j \geq K} p^{((n+\epsilon)(r-1)+1-s)nj}$, which converges when $s > (n+\epsilon)(r-1)+1+1/n$, and this bound is $< (c_0+\epsilon)(r-1)+1+1/c_0$.
	
	Since $\epsilon$ was arbitrary, to show $a(F^{\mathfrak{C}_{\Sigma(c_0)}}_r) \leq (c_0+\sup_{p^k}\rho(p^k))(r-1)+1$, it remains to show that $(c_0+\sup_{p^k}\rho(p^k))(r-1)+1 \geq c_0(r-1)+1+1/c_0$, or equivalently that $\sup_{p^k}\rho(p^k)(r-1) \geq 1/c_0$. We first find a bound for $\log(\prod_{i=1}^{c_0}(1-2^{-i})) \geq \sum_{i=1}^\infty \log(1-2^{-i})$. We can write $\log(1-2^{-i})$ as a Taylor series to get $\sum_{i=1}^\infty \log(1-2^{-i}) = -\sum_{i=1}^\infty \sum_{\iota=1}^\infty \frac{2^{-\iota i}}{\iota}$. Changing the order of summation, we get $-\sum_{\iota=1}^\infty \sum_{i=1}^\infty \frac{2^{-\iota i}}{\iota} = -\sum_{\iota=1}^\infty \frac{2^{-\iota}}{\iota(1-2^{-\iota})} = -\sum_{\iota=1}^\infty \frac{1}{\iota(2^\iota-1)} > -\sum_{\iota=1}^\infty \frac{1}{2^{\iota-1}} = -2$. So the whole sum is absolutely convergent, and we conclude $\log(\prod_{i=1}^{c_0}(1-2^{-i})) > -2$. Moreover, $\log(c_0!) > c_0 (\log(c_0)-1)$ by Stirling's approximation, so $\rho(2) > \frac{\log(c_0)-1}{(c_0-1)\log(2)} - \frac{2}{c_0 \log(2)} > \frac{\log(c_0)}{(c_0-1)\log(2)} - \frac{3}{c_0 \log(2)}$. 
	Therefore, because $c_0 > C_0 > 2e^3$, the result holds. 	
	
	We wish to show that $\sup_{p^k}\rho(p^k) = \rho(2) \leq \frac{\log(c_0)}{(c_0-1)\log(2)}$; we have already shown the lower bound for $\rho(2)$. For each $p^k$, we have the upper bound $\rho(p^k) < \frac{\log(k)+\log(c_0!)}{c_0(c_0-1)\log(p^k)} < \frac{\log(k)+\log(c_0)}{(c_0-1)\log(p^k)}$.
	Since $c_0 > C_0 > 8$, a calculation shows, for $p^k \geq 3$, $\rho(p^k) < \frac{\log(k) + \log(c_0)}{(c_0-1)\log(p^k)} < \frac{\log(c_0)-3}{c_0 \log(2)} < \rho(2) < \frac{\log(c_0)}{(c_0-1)\log(2)}$.
	
	Finally, for the lower bound on $a(F^{\mathfrak{C}_{Alt(c_0)}}_r)$, $(c_0+\rho(2))(r-1)+1$ is given by Theorem \ref{spacetime-bounds}.
\end{proof}

The reader may generate their own examples, but we mention the following generalisation of the previous two theorems, which follows from essentially the same proof without any additional work, as the reader may verify. Suppose, for each of the types $A_d, B_d, C_d, D_d, ^2A_d, ^2D_d$ of classical groups (where $d$ denotes the Lie rank), we pick some non-negative integers $d_A,d_B,d_B,d_D,d_{(^2A)},d_{(^2D)}$. Let $\mathfrak{C}$ be the NE-formation generated by some subset of the cyclic groups of prime order including $C_2$ and $C_3$, some subset of the sporadic groups, some subset of the exceptional groups of Lie type, some subset of the alternating groups of degree $\leq c_0$ including $Alt(c_0)$, some subset of the classical groups of type $A_d(q)$ (if $d_A > 0$) over all $d \leq d_A$ and all prime powers $q$ including $A_{d_A}(2)$, and similarly for the other classical groups (for the Steinberg groups $^2A_{d_{(^2A)}}(q^2)$ and $^2D_{d_{(^2D)}}(q^2)$, we assume $^2A_{d_{(^2A)}}(4)$ and $^2D_{d_{(^2D)}}(4)$ are in $\mathfrak{C}$). Including $C_3$ ensures that the normalisers of the Steinberg groups in $\mathfrak{C}$ in the corresponding general linear groups over $\F_4$ are also in $\mathfrak{C}$; including $C_2$ ensures $Sym(c_0) \in \mathfrak{C}$. $C_2$ and $C_3$ also ensure that Theorem \ref{thm:lowerbound-general} applies.

\begin{thm}
	Suppose $c_0 > C_0$. For any $r \geq 2$, $a(F^{\mathfrak{C}}_r) = c_{\mathfrak{C},\text{space}}(r-1)+1$, where 
	\begin{multline*}
	c_{\mathfrak{C},\text{space}} = \max\left\{c_{Alt(c_0)}+\frac{\log_2(c_0!)}{(c_0-\delta(c_0))(c_0-1)}, c_{A_{d_A}(2)}+\frac{\log_2(c_0!)}{(d_A+1)(c_0-1)}, \right. \\ c_{B_{d_B}(2)}+\frac{\log_2(c_0!)}{(2d_B+1)(c_0-1)}, c_{C_{d_C}(2)}+\frac{\log_2(c_0!)}{2d_C(c_0-1)}, c_{D_{d_D}(2)}+\frac{\log_2(c_0!)}{2d_D(c_0-1)}, \\ \left. c_{(^2A_{d_{(^2A)}})(4)}+\frac{\log_2(c_0!)}{2(d_{(^2A)}+1)(c_0-1)}, c_{(^2D_{d_{(^2D)}})(4)}+\frac{\log_2(c_0!)}{4d_{(^2D)}(c_0-1)}\right\}.
	\end{multline*}
\end{thm}

Note that analogous results can still be proved without the assumption that $A_{d_A}(2)$, etc., are in $\mathfrak{C}$, but in this case the constant $C_0$ may have to be increased, as a function of the smallest $p^k$ such that $A_{d_A}(p^k)$, etc., are in $\mathfrak{C}$. By contrast, $C_0$ does not depend on $d_A$, etc.

Up to error terms which tend to $0$ as $q \to \infty$, the values given in \cite[Table 5.1.A, Table 5.4.C]{KL} show that $c_{A_{d_A}(q)} = d_A+1$; $c_{B_{d_B}(q)} = d_B+1/(2d_B+1)$; $c_{C_{d_C}(q)} = d_C+1/2+1/(2d_C)$; $c_{D_{d_D}(q)} = d_D-1/2+1/(2n)$; $c_{(^2A_{d_{(^2A)}})(q^2)} = d_{(^2A)}/2+1/2$;  $c_{(^2D_{d_{(^2D)}})(q^2)} = d_{(^2D)}/2-1/4+1/(2n)$.

\section{Groups with arbitrary abscissae}
\label{sec:arbitrary}

In this section we will prove Theorem~\ref{thmABC:arbitrary}. In passing we mention that similar results for other types of zeta functions have been obtained by Kassabov \cite{Kassabov} and Klopsch-Piccolo (unpublished).

For $n$ a positive integer, let $\pi(n) \leq \log_2(n)$ denote the number of prime factors of $n$ with repetitions (for example, $\pi(2^k)=k$). In particular, we set $\pi(1)=0$.

\begin{thm}
	\label{arbitraryabs}
	Let $G_\alpha = \prod_{p > 3} \mathrm{SL}(2,p)^{\lfloor p^\alpha \rfloor}$, for any real $\alpha \geq 0$. Then $\zeta_{G_\alpha}(s)$ has abscissa of convergence $\alpha/2+1$.
\end{thm}
\begin{proof}[Proof of Theorem~\ref{thmABC:arbitrary}]
Let $p > 3$ be a prime.
	We first list a few general facts about the representations of $\mathrm{SL}(2,p)$ that we will use. Since $p>3$, the group $\mathrm{PSL}(2,p)$ is simple. It is known that (see \cite[Section 8]{Humphreys}) $\mathrm{SL}(2,p)$ has exactly one absolutely irreducible representation in each dimension $d \in \{1, \ldots, p\}$, over the field $\F_{p^j}$ for all $j$. Let $q \neq p$ be a prime. The group $\mathrm{SL}(2,p)$ has non-trivial absolutely irreducible representations over $\F_{q^j}$ only in dimensions $\geq (p-1)/2$ by \cite[Theorem 5.3.9]{KL}. By the Wedderburn-Artin theorem, the number of non-trivial absolutely irreducible representations of $\mathrm{SL}(2,p)$ over $\F_{q^j}$ of dimension $n \geq (p-1)/2$ is at most 
	\begin{equation}\label{4-bound-reps}
	 |\mathrm{SL}(2,p)|/n^2 \leq 4(p^3-p)/(p-1)^2 = 4p + 8 + 8/(p-1) \leq 4p + 10.
	\end{equation}
	
	Now fix $\alpha\ge 0$ and denote $G=G_\alpha$.	We will split the proof into two parts: showing that $a(G) \leq \alpha/2+1$ and showing that $a(G) \geq \alpha/2+1$.
		
	We start by fixing a finite field $\F_{q^j}$, $q$ prime, and count absolutely irreducible representations in dimension $n$. All copies of $\mathrm{SL}(2,p)$ with $p \neq q$ and $p > 2n+1$ must act trivially.

	By \cite{Fein} (see also \cite[Lemma 5.5.5]{KL}), we may write an $n$-dim\-en\-sio\-nal absolutely irreducible representation of $G$ as a tensor product of one absolutely irreducible $m$-dim\-en\-sio\-nal representation of $\mathrm{SL}(2,q)^{\lfloor q^\alpha \rfloor}$ and one absolutely irreducible  $n/m$-dim\-en\-sio\-nal representation of $H=\prod \mathrm{SL}(2,p)^{\lfloor p^\alpha \rfloor}$, where the product of $H$ ranges over the primes $p$ such that $p \leq 2n/m+1$ and $p \neq 2,3,q$. For the case $q \leq 3$, we simply assume $m=1$. As an upper bound for $r^\ast(H,\F_{q^j},n/m)$, note that any such representation must be a tensor product of at most $\log_2(n/m)$ non-trivial absolutely irreducible representations of one of the special linear groups in $H$. Additionally, the number of $\mathrm{SL}$-factors appearing in $H$ is at most
	\begin{align*}
		&\sum_{p \leq 2n/m+1}
	\lfloor p^\alpha \rfloor \leq (2n/m+1)^{\alpha+1},
\end{align*}
and by \eqref{4-bound-reps} each factor has at most $4(2n/m+1)+10 \leq 8n/m+14$ non-trivial absolutely irreducible representations. This gives the upper bound 
\[r^\ast(H,\F_{q^j},n/m) \leq ((2n/m+1)^{\alpha+1}(8n/m+14))^{\log_2(n/m)} = 2^{O_{\alpha}(\log_2(n/m)^2)},\] where the implied constant depends only on $\alpha$.

Now we estimate $r^\ast(\mathrm{SL}(2,q)^{\lfloor q^\alpha \rfloor},\F_{q^j},m)$ similarly: any such representation must be a tensor product of  non-trivial absolutely irreducible representations of  the $\mathrm{SL}(2,q)$ factors, of which there are $\lfloor q^\alpha \rfloor$, each with at most one absolutely irreducible representation in each dimension. Thus $r^\ast(\mathrm{SL}(2,q)^{\lfloor q^\alpha \rfloor},\F_{q^j},m)$ is bounded by the number of possible distributions of the $\pi(m)$ prime factors of $m$ to the $\lfloor q^\alpha \rfloor$ copies of $\mathrm{SL}(2,q)$. This gives the upper bound $r^\ast(\mathrm{SL}(2,q)^{\lfloor q^\alpha \rfloor},\F_{q^j},m) \leq q^{\alpha\pi(m)}$.

In total, summing over $m$, we have $r^\ast(G,\F_{q^j},n) = q^{\alpha\pi(n)}2^{O_{\alpha}(\log_2(n)^2)}$.

Putting all this together, 
we get the upper bound 
\begin{equation*}\label{eq:nq neq 3}
	\log(\zeta_G)(s) \leq \sum_{q \in \mathcal{P}}
	\sum_{n,j\in \mathbb{N}} \frac{q^{\alpha\pi(n)}2^{O_{\alpha}(\log_2(n)^2)}}{j} q^{-snj} \frac{q^{nj}-1}{q^j-1}.
\end{equation*}
For a fixed $k=nj$, because $q \geq 2$, we have $\frac{q^{nj}-1}{q^j-1} \leq \frac{q^{nj}-1}{q-1} \leq q^{nj-1}\frac{q}{q-1} \leq 2q^{k-1}$. So replacing $nj$ with $k$, we find the upper bound
\begin{equation*}
	2 \sum_{q \in \mathcal{P}} \sum_{k} 2^{O_{\alpha}(\log_2(k)^2)}q^{\alpha \pi(k)+(1-s)k-1},
\end{equation*}
which converges when $\alpha \pi(k)+(1-s)k < -\delta k$ for some $\delta > 0$, or equivalently when $\alpha \pi(k)+(1-s)k < 0$ for all integers $k\geq 1$. Indeed, summing over $k$, $\sum 2^{O_{\alpha}(\log_2(k)^2)}q^{-1-\delta k} = O(q^{-1-\delta'}/(1-q^{-\delta'}))$ for any $0 < \delta' < \delta$, for an implied constant independent of $q$, so the sum converges by the integral test.

In particular, it converges when $s > \alpha \max\{\pi(k)/k \mid k \in \mathbb{N}\} +1$. Using $\pi(k) \leq \log_2(k)$, it is easy to check that $\alpha(\log_2(k)/k)+1$ attains its upper bound $\alpha/2+1$ for $k=2$ and $k=4$, and is otherwise smaller. 
We conclude that $a(G) \leq \alpha/2+1$.

To show that $a(G) \geq \alpha/2+1$, we count representations of dimension $2$ over fields of prime order. Here we have $\log(\zeta_G)(s) \geq \sum_{q \geq 5} \lfloor q^\alpha \rfloor q^{-2s} (q+1)$. For large $q$, $\lfloor q^\alpha \rfloor \geq (1-\epsilon)q^\alpha$ for any $\epsilon > 0$, so this sum is at least $ (1-\epsilon)\sum_{q \geq Q} q^{\alpha+1-2s}+q^{\alpha-2s}$, which diverges when $\alpha+1-2s > -1$, or equivalently when $s < \alpha/2+1$. So $a(G) \geq \alpha/2+1$.
\end{proof}

\section{Finite extensions with large abscissae}
\label{largeabscissae}

Here we will construct examples of groups $G$ that are \emph{split extensions} 
of a finite index normal subgroup $N$ such that $a(G)> a(N)$.
We will use a product of groups of Lie type, acted on by 
the cyclic groups $C_f$  of order $f$ via Frobenius automorphisms.

\begin{thm}
Let $f \geq 5$ be a prime and let
$G = \prod_{p\in \mathcal{P}, p > 3} \SL(2,p^f)^{p^f}$. Let the cyclic group $C_f$ act diagonally on $G$ by Frobenius automorphisms on the factors. Then 
\[a(G) =  \frac{3}{2} - \frac{f-1}{4f} < \frac{3}{2} \leq a(G \rtimes C_f).\]
\end{thm}
Note that $a(G \rtimes C_f) - a(G) \geq (1-f^{-1})/4$, that is, one quarter of the maximum increase predicted by Proposition \ref{opensbgp} (see also Remark \ref{perfect}) as $f$ tends to infinity.
\begin{proof}
We copy the proof of Theorem \ref{thmABC:arbitrary} to show convergence. In characteristic $q \neq p$, the group $\SL(2,p^f)$ has non-trivial absolutely irreducible representations only in dimensions $\geq (p^f-1)/2$ when $p > 3$ by \cite[Theorem 5.3.9]{KL}, and the number of such representations over $\F_{q^j}$ of dimension $n \geq (p^f-1)/2$ is at most  $ |\SL(2,p^f)|/n^2 \leq 4(p^{3f}-p^f)/(p^f-1)^2 = 4p^f + 8 + 8/(p^f-1) \leq 4p^f + 9$.

Now we consider characteristic $p$. The splitting field for $\SL(2,p^f)$ is $\F_{p^f}$ by \cite[Proposition 5.4.4]{KL}, and we consider first the irreducible representations over a field $\F_{p^{fj'}}$. We can use \cite[Theorem 5.4.5]{KL} to see that $r^\ast(\SL(2,p^f),\F_{p^{fj'}},n)$ is the number of ways of writing $n$ as an ordered product of $f$ numbers between $1$ and $p$. From this description, we have $r^\ast(\SL(2,p^f),\F_{p^{fj'}},n) = f^{\pi(n)}$, where as before $\pi(n)$ is the number of prime factors of $n$.

On the other hand, consider $\F_{p^j}$ with $j$ coprime to $f$. By \cite[Proposition 5.4.6]{KL}, and the description of the absolutely irreducible representations of $\SL(2,p)$ in \cite[Section 8]{Humphreys} (compare with the proof of Theorem \ref{thmABC:arbitrary}), there is one absolutely irreducible representation of $\mathrm{SL}(2,p^f)$ in each of dimensions $1^f, 2^f, \ldots, p^f$, and no others.

Now we can apply these upper bounds to count absolutely irreducible representations for the whole of $G$. Fix a finite field $\F_{q^j}$, $q$ prime, and count absolutely irreducible representations in dimension $n$.

As for Theorem \ref{thmABC:arbitrary}, we write such a representation as a tensor product of two absolutely irreducible representations, one of $\SL(2,q^f)^{q^f}$ of dimension $m$ and one of $H = \prod_{p^f \leq 2n/m+1, p \neq 2,3,q} \mathrm{SL}(2,p^f)^{p^f}$ of dimension $n/m$; when $q=2$ or $3$, we simply assume $m=1$. We see that $H$ contains $\sum_{p^f \leq 2n/m+1, p \neq 2,3,q} p^f \leq (2n/m+1)(n/m+1)$ special linear direct factors, each with at most $8n/m+13$ absolutely irreducible representations in any dimension $\leq n/m$; we conclude that 
\[r^\ast(H,\F_{q^j},n/m) \leq ((2n/m+1)(n/m+1)(8n/m+13))^{\log_2(n/m)} = 2^{O(\log_2(n/m)^2)}.\]

When $j$ is coprime to $f$, as for Theorem \ref{thmABC:arbitrary}, we get $r^\ast(\SL(2,q^f)^{q^f},\F_{q^j},l^f) \leq (q^f)^{\pi(l)}$; meanwhile $r^\ast(\SL(2,q^f)^{q^f},q^j,m) = 0$ for other values of $m$.

When $j = fj'$ for an integer $j'$, each $\SL(2,q^f)$ has at most $f^{\pi(m)}$ absolutely irreducible representations over $\F_{q^{fj'}}$ in dimension $m$, giving the upper bound $r^\ast(\SL(2,q^f)^{q^f},\F_{q^{fj'}},m) \leq (q^ff)^{\pi(m)}$.

We now split $\log(\zeta_G) = \Sigma_1 + \Sigma_2$ into two sums, and consider the convergence of each separately, where
\begin{align*}
	\Sigma_1 &= \sum_{q \in \mathcal{P}} \sum_{n=1}^\infty \sum_{j:\text{gcd}(j ,f)=1} \frac{r^\ast(G,q^j,n)}{j} q^{-snj} |\mathbb{P}^{n-1}(\F_{q^j})|, \\
	\Sigma_2 &= \sum_{q \in \mathcal{P}} \sum_{n=1}^\infty \sum_{j'=1}^\infty \frac{r^\ast(G,q^{fj'},n)}{fj'} q^{-snfj'} |\mathbb{P}^{n-1}(\F_{q^{fj'}})|.
\end{align*}

First we deal with $\Sigma_1$. Write $n = m (n/m)$ as above. We may assume $m$ has the form $l^f$. We consider 
separately the case where $l=1$: the reader may verify it converges for all 
$s>1$. So we need only check convergence for $$\Sigma_1' = \sum_{q \in 
	\mathcal{P}} \sum_{j:\text{gcd}(j ,f)=1} \sum_{l \geq 2} \sum_{l'=1}^\infty 
\frac{r^\ast(G,\F_{q^j},l^fl')}{j} q^{-sl^fl'j} |\mathbb{P}^{l^fl'-1}(\F_{q^j})|,$$ 
where we write $n$ as $l^fl'$. The approximations above give
 $$\Sigma_1' \leq 
\sum_{q \in \mathcal{P}} \sum_{\text{gcd}(j,q)} \sum_{l \geq 2} 
\sum_{l'=1}^\infty 2^{O(\log_2(l')^2)} q^{f\pi(l)} q^{(1-s)l^fl'j-1}.$$
 Now we sum over all $l$, $l'$ and $j$ such that $l^fl'j=k$ to get
  $$\Sigma_1' \leq \sum_{q \in \mathcal{P}} \sum_{k \geq 2^f} 2^{O(\log_2(k)^2)} 
q^{\pi(k)+(1-s)k-1}$$
 (because $k \geq l^f \geq 2^f$) -- which, as for 
Theorem \ref{thmABC:arbitrary}, converges when $\pi(k)+(1-s)k < 0$ for all $k \geq 2^f$, or equivalently (using $\pi(k) \leq \log_2(k)$) when $s > \max_{k \geq 2^f} (\pi(k)/k)+1 = 1 + f/2^f$.

Now we will deal with $\Sigma_2$. From the bounds above, we get $r^\ast(G,\F_{q^{fj'}},n) \leq 2^{O(\log_2(n)^2)} (q^ff)^{\pi(n)}$, and we conclude $$\Sigma_2 \leq \sum_{q \in \mathcal{P}} \sum_{n,j'=1}^\infty 2^{O(\log_2(n)^2)} q^{f\pi(n)+(1-s)nfj'-f}.$$ Replacing 
$nj'$ with $k$ as in the proof of Theorem \ref{thmABC:arbitrary}, by the arguments above, this converges for $s$ such that $f(\pi(k)+(1-s)k-1) < -1$ for all $k$, which holds when $s > \pi(k)/k-(f-1)/(fk)+1$.

This function has a maximum when $k=4$, so $\Sigma_2$ converges when $s > 3/2 - (f-1)/(4f)$.
 For all primes $f \geq 5$,  we find 
 \[
 	a(G) \leq \max\Bigl(\frac{3}{2} - \frac{f-1}{4f}, 1 + \frac{f}{2^f}\Bigr) = \frac{3}{2} - \frac{f-1}{4f} < \frac{3}{2}.\]

We obtain a corresponding lower bound by looking at representations of dimension $4$. We have $r^\ast(G,\F_{q^f},4) \geq q^{2f}$ for $q \geq 5$, so 
\[\log(\zeta_G)(s) \geq \sum_{q \in \mathcal{P}} \frac{q^{2f}}{f} q^{-4fs} \frac{q^{4f}-1}{q^f-1} \geq \sum_{q \in \mathcal{P}} \frac{q^{(5-4s)f}}{f},\]
which diverges when $(5-4s)f \geq -1$, or equivalently when $s \leq 3/2 - (f-1)/(4f)$. Therefore $a(G) = 3/2 - (f-1)/(4f)$.

On the other hand, when we extend $G$ by $C_f$ acting diagonally by Frobenius automorphisms on each $\SL(2,p^f)$, each $\SL(2,p^f) \rtimes C_f$ has an absolutely irreducible representation of dimension $2f$ over $\F_p$. So we get $\log(\zeta_{G \rtimes C_f})(s) \geq \sum_q q^f q^{-2fs} \frac{q^{2f}-1}{q-1} \geq \sum_q q^{f-2fs+2f-1}$, which diverges when $3f-2fs \geq 0$, or equivalently $s \leq 3/2$.
\end{proof}

\section{Zeta functions of finite groups}\label{sec:finite}

In this section we prove Theorem~\ref{thmABC:finite_groups}.
\begin{proof}[Proof of Theorem~\ref{thmABC:finite_groups}]
Since $\Q[G]$ is a semisimple $\Q$-algebra, we may write
\[
	\Q[G] \cong \prod_{i=1}^r A_i
\]
where $A_i$ is a simple $\Q$-algebra with centre $K_i = Z(A_i)$. Let $\mathcal{O}_i$ denote the ring of algebraic integers in $K_i$.
Put $n_i^2 = \dim_{K_i} A_i$. Note that $\dim_{K_i} A_i \leq \dim_\Q \Q[G] = |G|$ and so $n_i \leq \sqrt{|G|}$. We have $r = |\Irr(G,\Q)|$ and the rational irreducible representations of $G$ correspond to the simple modules of the algebras $A_i$.
Let $\chi$ be the character of the simple rational representation of $\Q[G]$, which factors through $A_i$. Let $m(\chi)$ denote the Schur index. Then 
$\chi(1) = n_i m(\chi) [K_\chi:\Q]$; i.e. $n_i = \frac{\chi(1)}{[K_\chi:\Q] m(\chi)}$.

We say that a prime number $p$ in \emph{unramified} in $\Z[G]$, if the following hold
\begin{enumerate}
\item $p$ is unramified in $K_i$ for all $i\in \{1,\dots, r\}$,
\item for all $i$ the algebra $A_i$ splits at all primes $\mathfrak{p}$ of $\mathcal{O}_i$ dividing $p$
\item $\Z_p[G]$ is a maximal $\Z_p$-order in
$\Q_p[G]$.
\end{enumerate}
The integral group ring $\Z[G]$ is a $\Z$-order in $\Q[G]$. For almost all prime numbers $p$ the completion $\Z_p[G]$ is a maximal $\Z_p$-order in
$\Q_p[G]$; this follows from the existence of maximal orders and \cite[(11.6)]{Reiner}.
We deduce that almost all primes are unramified in $\Z[G]$; see \cite[(8.4)]{Neukirch} and \cite[(32.1)]{Reiner}.
If $p$ is unramified, then 
\[
	\Q_p[G] \cong \prod_{i=1}^r \prod_{\mathfrak{p} \mid p} M_{n_i}(K_{i,\mathfrak{p}}).
\]
By \cite[(17.3)]{Reiner} the maximal $R$-orders in $M_n(F)$ are conjugate to $M_n(R)$ if $R$ is a complete discrete valuation ring with quotient field $F$.
In particular, we deduce for unramified $p$ that
\[
	\Z_p[G] \cong \prod_{i=1}^r \prod_{\mathfrak{p} \mid p} M_{n_i}(\mathcal{O}_{i,\mathfrak{p}})
\]
and so
\[
	\F_p[G] \cong \prod_{i=1}^r \prod_{\mathfrak{p} \mid p} M_{n_i}(\mathcal{O}_{i}/\mathfrak{p}).
\]
Let $N(\mathfrak{p}) = \vert\mathcal{O}_i/\mathfrak{p}\vert$ be the norm of the prime ideal $\mathfrak{p} \subseteq \mathcal{O}_i$ and write
$N(\mathfrak{p}) = p^{f(\mathfrak{p})}$.
We recall that $r^*(M_{n_i}(\mathcal{O}_i/\mathfrak{p}), \F_{p^k},n)$ vanishes unless $n = n_i$ and  $k$ is a multiple of $f(\mathfrak{p})$, in which case $r^*(M_{n_i}(\mathcal{O}_i/\mathfrak{p}), \F_{p^k},n_i) = f(\mathfrak{p})$. This allows us to calculate the zeta function of $G$ up to ramified primes. The correction amounts to multiplication of a finite number of factors of the form \eqref{eq:rational}. We obtain
\begin{align*}
	\zeta_G(s) &\sim \prod_{i=1}^{r}\exp\left(\sum_{\mathfrak{p} \subseteq \mathcal{O}_i} \sum_{k=1}^\infty\frac{r^*_{n}(M_{n_i}(\mathcal{O}_i/\mathfrak{p}), \F_{p^k})}{k} p^{-sn_ik} |\mathbb{P}^{n_i-1}(\F_{p^k})|\right)\\
	&= \prod_{i=1}^{r} \exp\left(\sum_{\mathfrak{p} \subseteq \mathcal{O}_i}\sum_{k'=1}^\infty\frac{f(\mathfrak{p})}{k'f(\mathfrak{p})} p^{-sn_ik'f(\mathfrak{p})} |\mathbb{P}^{n_i-1}(\F_{p^{k'f(\mathfrak{p})}})|\right)\\
	&= \prod_{i=1}^{r} \prod_{\mathfrak{p} \subseteq \mathcal{O}_i} \exp\left(\sum_{k'=1}^\infty \frac{1}{k'} N(\mathfrak{p})^{-sn_ik'} \bigl(N(\mathfrak{p})^{k'(n_i-1)} + \dots + N(\mathfrak{p})^{k'} + 1\bigr)\right)\\
	&= \prod_{i=1}^{r} \prod_{\mathfrak{p} \subseteq \mathcal{O}_i} \prod_{j=0}^{n_i-1} \bigl(1- N(\mathfrak{p})^{-sn_i+j}\bigr)^{-1}
	 = \prod_{i=1}^{r} \zeta_{K_i}^{\# n_i}(s).
\end{align*}
It is a classical result of Hecke that the Dedekind zeta function $\zeta_K$ admits a meromorphic extension to $\C$ with a simple pole at $s = 1$.
The poles of $\zeta_{K}^{\# n}(s)$ are thus $\frac{1}{n}, \frac{2}{n}, \dots, 1$. The rational correction factors \eqref{eq:rational} satisfy $a_k \leq \sqrt{|G|}$ and thus can have poles at non-negative rational numbers which
are also bounded from above by $1 - \sqrt{|G|^{-1}}$.
\end{proof}

\section{Virtually abelian groups}\label{sec:rationality}

In this section we prove Theorem~\ref{thmABC:rational}.
\begin{proof}[Proof of Theorem~\ref{thmABC:rational}]
Let $A \trianglelefteq G$ be an abelian normal subgroup of finite index $d = |G:A|$ in $G$. 
Then every absolutely irreducible representation of $G$ has dimension at most $d$.
For every $n \in \{1,2,\dots,d\}$ there is a moduli variety $M_{n}$ defined over $\F_p$ such that
the closed points $M_n(\F_{p^k})$ are in bijective correspondence with the isomorphism classes of absolutely irreducible representations of $G$ over the field $\F_{p^k}$; see e.g.~\cite[Theorem 6.23]{Korthauer}.
Define $V_n = M_n \times \mathbb{P}^{n-1}$.
Then 
\[
	|V_n(\F_{p^k})| = r^*(G,\F_{p^k},n)\cdot |\mathbb{P}^{n-1}(\F_{p^k})|
\]
hold for all $k \in \mathbb{N}$.
This implies that
\[
	\zeta_{G,p}(s) = \sum_{n=1}^{d} Z(V_n,s)
\]
where $Z(V_n,s)$ denotes the local Hasse-Weil zeta function of $V_n$. The local Hasse-Weil zeta function is a rational function in $p^{-s}$ by \cite{Dwork}.
\end{proof}

The following is probably well-known and it was communicated to the authors by Alexander Moret\'o.

\begin{lem}[\cite{moreto}]
 Let $K$ be a field of characteristic $p > 0$. Then there exists a real-valued function $f$ such that if $G$ is a finite group such that all irreducible representations over $K$ have dimension at most $n$, then $G$ has a characteristic $p$-abelian subgroup $A$ such that $\vert G : A\vert \leq f(n)$.
\end{lem}

It follows that the groups to which the proof of Theorem~\ref{thmABC:rational} can be applied are exactly the finitely generated virtually abelian groups. However, the class of UBERG groups with rational local factors is larger; for instance, it contains the lamplighter group (see \ref{sec:lamplighter}). It would be interesting to have a description of the class of groups for which the rationality result holds.

\appendix

\section{Examples}\label{app:A}
In this section we will collect various examples of zeta functions that we have calculated. Some examples have appeared in this article; others can be obtained by following the steps of Example~\ref{ex:ZwrC2}. Finally, we also calculate the zeta functions for the lamplighter groups $C_2\wr \mathbb{Z}$ and $C_3\wr \mathbb{Z}$. In this section $\zeta(s)$ will denote the Riemann zeta function.

\subsection{Abelian groups}
Since one-dimensional representations are always absolutely irreducible, $\zeta_G$ can be easily calculated for abelian groups by counting homomorphisms into $\F_q^\times$ for every $q$.
\begin{ex}[$C_p$ cyclic group]
  $$ \zeta_{C_p}(s) = \zeta(s) \cdot \prod_{p} \prod_{\chi\in \F_p^\times} \left( 1- \frac{\chi(p)}{p^s}\right)^{-1}  = 
    \zeta(s) \cdot  \zeta_{\mathbb{Q}(\eta)}(s) \cdot \left(1-\frac{1}{p^s}\right) $$
    where $\eta$ denotes a primitive $p$-th root of unity.
\end{ex}

\begin{ex}[$\mathbb{Z}^r$ free abelian group]
$$\zeta_{\hat{\Z}^r}(s) = \prod_{i=0}^r \zeta(s-i)^{(-1)^{r-i} \binom{r}{i}}$$
\end{ex}

\subsection{Symmetric groups}
In Theorem~\ref{thmABC:finite_groups} we gave a formula for the zeta 
function of a finite group up to rational factors. An exact formula for $\zeta_G$ requires a concise understanding of the modular representation theory of $G$.
The modular representation theory of the symmetric group $S_n$ is 
well-studied
and the description of the absolutely irreducible representations of $S_n$ as quotients $S^{\mu}/(S^{\mu}\cap (S^{\mu})^\perp)$ of the Specht modules $S^{\mu}$ given in \cite[Theorem 4.9]{James1978} can be used to compute an exact formula for $\zeta_{S_n}$ for small values of $n$.  
\begin{ex}[$S_4$ symmetric group]
\[ \zeta_{S_4}(s) = \zeta(s)^2 \zeta^{\#2}(s) \zeta^{\#3}(s)^2 (1-2^{-s}) \prod_{k\in \{2,3\}}\prod_{ j=0}^{k-1} (1-3^{j-ks}). \]
\end{ex}

We can also give a more concise
formula for the zeta function up to rational factors for primes below $n$. 
Given two meromorphic functions $f, g$ on $\mathbb{C}$ we write
$f \sim_n g$ if there is a rational function $h$ in $\{p^{-s} \mid  p \le 
n,\ p \text{ prime}\}$ such that $fh = g$.

\begin{ex}[$S_n$ symmetric group]
 Let $n \ge 2$ be an integer. Then 
 \[ 
   \zeta_{S_n}(s) \sim_n \prod_{\chi\in \mathrm{Irr}(S_n,\mathbb{C})} 
\prod_{j=0}^{\chi(1) -1} \zeta(\chi(1) s -j).
 \]
It follows from Theorem \ref{thmABC:finite_groups} that the zeta function $\zeta_{S_n}(s)$ has a pole of
order $P(n)$ at $s = 1$, where $P(n)$ denotes the number of partitions of $n$.
\end{ex}

\subsection{Virtually abelian groups}
For virtually abelian groups $\zeta_G$ can be calculating by inducing irreducible representations from an abelian normal subgroup. We explain this for $\mathbb{Z}\wr C_2$.
All the examples of this section can be calculated in essentially the same way. 

\begin{ex}[$\mathbb{Z}\wr C_2$ wreath product]\label{ex:ZwrC2}
Consider the group $G=\mathbb{Z}\wr C_2 = \mathbb{Z}^2 \rtimes C_2$, where $C_2$ swaps the two copies of $\Z$.
Then 
$$ \zeta_{\mathbb{Z}\wr C_2}(s) \sim \frac{\zeta(s-1)^2}{\zeta(s)^2} \cdot \frac{\zeta(2s)\zeta(2s-3)}{\zeta(2s-1)\zeta(2s-2)}.$$ 
 It is easy to see that $G^\mathrm{ab} \cong C_2 \times \mathbb{Z}$. Therefore, if $2\nmid q$, $r^\ast(G,\mathbb{F}_{q},1) = 2 (q -1)$ and, if $2\mid q$,  $r^\ast(G,\mathbb{F}_{q},1) = q -1$.
 
 For degree two representations, consider two distinct one-dimensional irreducible representations $\chi_1$, $\chi_2$ of $\mathbb{Z}$ over $\F_q$, then  the induced representation $\rho(\chi_1,\chi_2) = \Ind_{\Z^2}^G(\chi_1 \otimes \chi_2)$  is an absolutely irreducible representation of $G$ and the number of such representations is $ \frac{1}{2}(q -1)(q -2)$ (the order of $\chi_1,\chi_2$ does not matter). 
However, it is possible that the field of definition of the induced representation $\rho$ is smaller than $\F_q$ and this is actually one of the main difficulties that arise for virtually abelian groups. In the case at hand, there could be a representation $\rho: G\to \mathrm{GL}_2(\F_{q})$ such that $\rho|_{\Z^2}$ is diagonalisable over $\F_{q^{2}}$, but not over $\F_q$. More precisely, given $\chi_1,\chi_2 \colon \Z \to \F_{q^2}$, then
$\rho(\chi_1,\chi_2)$ is defined over $\F_q$ exactly if $\chi_1$ and $\chi_2$ are conjugate under $\mathrm{Gal}(\F_{q^{2}}/ \F_q)$.
The number of representations of this type is $\frac{1}{2}(q^{2}-q)$ and
\[ r^*(G,\mathbb{F}_{q},2) = \frac{1}{2}(q -1)(q -2) +  \frac{1}{2}(q^{2}-q) = q^{2} - 2 q +1;\]
we note that the  factor $\frac{1}{2}$ disappeared. 

Calculating the zeta function $\zeta_{G}(s)$:
	\begin{multline*}
		\exp \left( \sum_{p\neq 2} \left(\sum_{j=1}^\infty \frac{ 2(p^j -1) }{j} p^{-sj} + \sum_{j=1}^\infty \frac{p^{2j} - 2 p^j +1}{j} p^{-2sj} (p^j +1)\right) \right) \cdot \\  \exp \left(  \sum_{j=1}^\infty \frac{ 2^j -1 }{j} p^{-sj} + \sum_{j=1}^\infty  \frac{2^{2j} -2 \cdot 2^j +1)}{j} p^{-2sj} (2^j +1) \right)   \\ = \left(1-\frac{1}{2^{s-1}}\right) \left(1-\frac{1}{2^s}\right)^{-1}  \frac{\zeta(s-1)^2}{\zeta(s)^2} \cdot \frac{\zeta(2s) \zeta(2s-3)}{\zeta(2s-1) \zeta(2s-2)}.
	\end{multline*}
\end{ex}

\begin{ex}[$D_\infty$ infinite dihedral group]
 $$ \zeta_{D_\infty}(s) \sim \zeta(s)^4 \cdot 
\frac{\zeta(2s-2)}{\zeta(2s-1) \zeta(2s)^{2}}$$ 
\end{ex}

\begin{ex}[$\mathbb{Z}\wr C_3$ wreath product]
  $$ \zeta_{\mathbb{Z}\wr C_3}(s) \sim \frac{\zeta(s-1)^3}{\zeta(s)^3} \cdot \frac{\zeta(3s-3)\zeta(3s-2)}{\zeta(3s)}$$
\end{ex}

\begin{ex}[$BS(1,-1)$ metabelian Baumslag-Solitar group]
 $$ \zeta_{BS(1,-1)}(s) \sim \frac{\zeta(s-1)^2}{\zeta(s)^2} \cdot \frac{\zeta(2s)^2\zeta(2s-3)}{\zeta(2s-1)\zeta(2s-2)^2}$$ 
\end{ex}

\subsection{Two lamplighter groups}\label{sec:lamplighter}

 Let $G$ be the lamplighter group $C_2 \wr \Z$. Let $N = 
\bigoplus_{i\in \Z} C_2$ denote the base of the wreath product. We determine the 
number $r^\ast (G, \F_q,n)$ absolutely irreducible $n$-dimensional 
representations of $G$ over the finite field $\F_q$.
 
\begin{prop}\label{prop:lamp_rep}
 Let $q$ be a prime power. 
 \begin{enumerate}[(i)]
  \item If $q$ is even, then all irreducible $\F_q$-representations of $G$ are $1$-dimensional and $r^\ast(G, \F_q,1) = q - 1$.
  \item If $q$ is odd, then $$r^\ast(G, \F_q,n) = (q - 1)\frac{F(n)}{n} $$ where $F (n) = \sum_{d\mid n} 2^d \mu(n/d)$. 
 \end{enumerate}
 \end{prop}
 Before we can prove the proposition, we study the function $F$.
 
 \begin{lem}\label{lem:F}
  Let $n \in \mathbb{N}$. Then $F (n) = \sum_{d\mid n} 2^d \mu(n/d)$ is the number of sequences in $\{\pm1\}^\Z$ which have stabilizer $n\Z$ under the Bernoulli shift action.
 \end{lem}
\begin{proof}
  For every divisor $d$ of $n$ let $X(d)$ denote the subset of $\{\pm1\}^\Z$ consisting of sequences with stabilizer $d\Z$ under the shift action of $\Z$. Define $F(d) = \vert X(d)\vert$; we will show that this function satisfies the formula. The union $\bigcup_{d\mid n} X(d)$ consists of all sequences which are stabilized by $n\Z$. Since the first $n$ entries determine such a sequence completely, we deduce 
  \begin{equation}\label{eq:inversion_form}
    2^n = \sum_{d\mid n} \vert X(d)\vert = \sum_{d\mid n} F (d).  
  \end{equation}
 The M\"obius inversion formula implies that 
 \[
   F (n) = \sum_{d\mid n}2^d \mu(n/d). \qedhere
 \]
 \end{proof}

 \begin{proof}[Proof of Proposition~\ref{prop:lamp_rep}] Let $V$ be an absolutely irreducible representation of $G$ over $\F_q$. If $q$ is even, then $V_{\vert N}$ is trivial, since $N$ is a $2$-group and $V _{\vert N}$ is semisimple. Every irreducible representation over $\F_q$ factors through the infinite cyclic quotient. In particular, the representations are $1$-dimensional and $r_1^\ast (G, q) = q - 1$.
 
 Assume that $q$ is odd. Then $V_{\vert N}$ decomposes as a sum of irreducible $\F_q[N]$-modules. We note that every irreducible $N$-module is absolutely irreducible, one dimensional and of the form 
 \[\chi_v ((\sigma_i)_{i\in\Z}) = \prod_{i} v_i^{\sigma_i}\]
 for some sequence $v = (v_i )_{i\in\Z} \in \{\pm1\}^\Z$. As $V_{\vert N}$ is finite dimensional and consists of a single orbit of a representation $\chi_v$, the sequence $v$ needs to be periodic with some period $n$. The representation factors through $G_n := \bigoplus_{i\in\Z/n\Z} C_2 \rtimes \Z$. The central subgroup $n\Z$ in $G_n$ acts with a character $\psi : n\Z \to \F_q^\times$ and by Clifford Theory 
 $$ V \cong \mathrm{Ind}^G_{N \times n\Z} (\chi_v \otimes \psi).$$ In particular, $V$ has dimension $n$. Conversely, every irreducible representation is of this form. A sequence $v \in \{\pm1\}^\Z$ is periodic with period $n$ exactly if $\chi_v$ has stabilizer $n\Z$ in the infinite cyclic subgroup of $G$. By Lemma~\ref{lem:F}, the number of such sequences $G$ is $F(n)$. For every character $\psi : n\Z \to \F_q$, $\mathrm{Ind}_{N \times n\Z}^G (\chi_v \otimes \psi)$ is an absolutely irreducible representation of $G$ defined over $\F_q$. Since the orbit of $\chi_v$ contains exactly $n$ elements, we obtain $$r^\ast (G, q,n) = (q - 1) \frac{F(n)}{n}.$$ 
\end{proof}

\begin{thm} Let $G$ be the lamplighter group $C_2\wr \mathbb{Z}$. Then 
\[
   \zeta_G(s) = \left( \frac{1 - 2^{-s}}{ 1 - 2^{1-s}} \right) \cdot \prod_{p>2} \left(\frac{1 - 2p^{-s}}{1 - 2p^{1-s}}\right).
\]
The abscissa of convergence is $a(G) = 2$. 
\end{thm}
\begin{proof}
 For $p = 2$ the local zeta function is $$\log \zeta_{G,2}(s) = \log(1 - 2^{-s} ) - \log(1 - 2^{1-s} )$$ using the same calculation as for the infinite cyclic group. Let $p$ be an odd prime. Then $\log\zeta_{G,p}(s) $ equals
 \begin{align*} &  \sum_{n=1}^\infty \sum_{k=1}^\infty \frac{r_n^\ast (G, p^k )}{k} p^{-skn} \vert \mathbb{P}^{n-1}(\F_{p^k}) \vert  
 = \sum_{n=1}^\infty \sum_{k=1}^\infty \frac{(p^k -1) F(n)}{nk} p^{-skn} \vert \mathbb{P}^{n-1}(\F_{p^k}) \vert \\ &= \sum_{n=1}^\infty \sum_{k=1}^\infty \frac{(p^{nk}-1) F(n)}{nk} p^{-skn}  
 = \sum_{m=1}^\infty \sum_{n\mid m} \frac{(p^{m}-1) F(n)}{m} p^{-sm}  \\ & = \sum_{m=1}^\infty  \frac{(p^{m}-1) }{m} p^{-sm} \left( \sum_{n\mid m} F(n)\right)  
 = \sum_{m=1}^\infty  \frac{(p^{m}-1) 2^m }{m} p^{-sm} \\ &= \log(1-2 p^{-s}) - \log(1-2 p^{1-s}). 
 \end{align*}
Note that we used \eqref{eq:inversion_form} in the chain of equalities above.
 \end{proof}

 The calculation for the lamplighter group $C_3\wr \mathbb{Z}$ is very similar, 
albeit more involved, and here we only report the result of our calculation.

\begin{thm} 
Let $G$ be the lamplighter group $C_3\wr \mathbb{Z}$. Then 
 \[
	\zeta_{G}(s) = \frac{1-3^{-s}}{1-3^{1-s}} \prod_{p\equiv 1 \bmod 3 } 
\frac{1-3p^{-s}}{1-3p^{1-s}} \prod_{p\equiv 2 \bmod 3 } 
\frac{(1-3p^{-2s})(1+p^{1-s})}{(1-3p^{2-2s})(1+p^{-s})}.
\]
In particular, the abscissa of convergence of $C_3\wr \Z$ is $2$.
\end{thm}

\section{Comparing two zeta functions}\label{app:B}

What is the effect of using absolutely irreducible representations in the definition of the zeta function? That is, UBERG can be measured in terms of the growth of irreducible representations, instead of absolutely irreducible ones. What would happen if we defined our zeta function in those terms?

Let $\eta_G$ be the complex function defined by $$\log(\eta_G)(s) = \sum_{p \text{ prime}} \sum_{n=1}^\infty \frac{r_n(G,\F_{p^j})}{j} p^{-snj} |\mathbb{P}^{n-1}(\F_{p^j})|,$$ i.e. just replacing $r^\ast$ with $r$ in the definition of $\zeta_G$.

Clearly $\zeta_G(s) \leq \eta_G(s)$ for real $s$, where they both converge. But in general, they need not have the same abscissa of convergence. We omit the proof of the following theorem for conciseness, as it follows the lines of the proof of Theorem~\ref{thmABC:arbitrary}.

\begin{thm}
	Let $G = \prod_{p \text{ prime } \geq 3} SL(2,p^p)^{p^p}$. Then $\eta_G$ has abscissa of convergence $\geq 3/2$, while $\zeta_G$ has abscissa of convergence $\leq 11/8$.
\end{thm}

As we saw in Sections~\ref{sec:finite} and \ref{sec:rationality}, for several 
groups we have a nice form for the zeta function, with properties like 
meromorphic continuation and rationality of local factors. In contrast, 
the next example illustrates that the zeta function $\eta_G(s)$ defined via irreducible 
representations can be `wild' even for virtually abelian groups.

\begin{ex}
Let $G = \Z\wr C_2$.
	By repeating word for word the calculations in example \ref{ex:ZwrC2} with 
	\emph{all} irreducible representations, we get 
	\[
	  \eta_G(s) \sim  \frac{\zeta(s-1)^2}{\zeta(s)^2} \cdot 
	\frac{\zeta(2s)}{\zeta(2s-2)}  \sqrt{\frac{\zeta(2s-3)}{\zeta(2s-1)}}.
	\]
	As we can see, this function does not have rational local factors and its analytic properties
	are more difficult to understand compared to the ones of $\zeta_G(s)$.
\end{ex}

\end{document}